\documentclass[11pt]{amsart}
\usepackage{amssymb,amsmath,amsthm,amsfonts,amsopn,url,hyperref,enumerate,mathtools,microtype,MnSymbol}
\usepackage[dvipsnames]{xcolor}
\usepackage[mathscr]{euscript}
\usepackage[normalem]{ulem}
\usepackage[all]{xy}
\usepackage{amscd}
\usepackage{tikz,tikz-cd}
\usepackage{comment}
\usepackage{mabliautoref}
\usepackage{xspace}
\usepackage{fullpage}
\usepackage{stmaryrd}

\newcommand{\myBrianconSkoda}{{Brian\c{c}on-Skoda}\xspace}
\newcommand{\DuBois}[1]{{\underline \Omega {}^0_{#1}}}
\newcommand{\field}[1]{\mathbb{#1}}

\newcommand{\Z}{\field{Z}}

\newcommand{\bZ}{\field{Z}}
\newcommand{\bP}{\field{P}}
\newcommand{\bC}{\field{C}}

\newcommand{\C}{\field{C}}

\newcommand{\bL}{\mathbb{L}}
\newcommand{\fm}{\mathfrak{m}}

\newcommand{\ideal}[1]{\mathfrak{#1}}
\newcommand{\m}{\ideal{m}}

\newcommand{\ffunc}[1]{\mathrm{#1}}
\newcommand{\func}[1]{\mathrm{#1} \,}

\newcommand{\Spec}{\func{Spec}}

\newcommand{\Ext}{\ffunc{Ext}}

\newcommand{\cf}{{\itshape{cf.} }}

\DeclareMathOperator{\Supp}{Supp}

\newcommand{\arrow}[1]{\stackrel{#1}{\rightarrow}}

\DeclareMathOperator{\myR}{{\mathbb R}}

\DeclareMathOperator{\myL}{{\mathbb L}}
\DeclareMathOperator{\reduced}{{red}}
\DeclareMathOperator{\sh}{{sh}}

\DeclareMathOperator{\perfd}{{perfd}}

\DeclareMathOperator{\Hom}{Hom}

\newcommand{\sHom}{{\mathcal{H}\mathrm{om}}}

\newcommand{\be}{\begin{enumerate}}
\newcommand{\ee}{\end{enumerate}}

\newcommand{\li}
 {\leftfootline}

\newcommand{\cO}{\mathcal{O}}

\renewcommand{\phi}{\varphi}

\newcommand{\Cech}{{\v{C}ech} }

\newcommand{\DBSplinter}{{reduced blowup-square splinter}\xspace}
\newcommand{\DBSplinters}{{reduced blowup-square splinters}\xspace}

\DeclareMathOperator{\KH}{KH}
\DeclareMathOperator{\Hir}{Hir}
\DeclareMathOperator{\Bir}{Bir}

\DeclareMathOperator{\Sym}{Sym}
\newcommand{\mydot}{{{\,\begin{picture}(1,1)(-1,-2)\circle*{2}\end{picture}\,}}}

\newcommand{\cl}{{\mathrm{cl}}}
\let\int\relax
\DeclareMathOperator{\int}{i}

\DeclareMathOperator{\coh}{coh}

\DeclareMathOperator{\epf}{epf}
\DeclareMathOperator{\ep}{ep}

\author{Linquan Ma}
\address{Department of Mathematics, Purdue University, West Lafayette, IN 47907, USA}
\email{ma326@purdue.edu}

\author{Peter M. McDonald}
\address{Department of Mathematics\\ Simon Fraser University \\ Burnaby, BC V5A 1S6}
\email{pma94@sfu.ca}

\author{Rebecca R.G.}
\address{Department of Mathematical Sciences \\ George Mason University \\ Fairfax, VA  22030}
\email{rrebhuhn@gmu.edu}

\author{Karl Schwede}
\address{Department of Mathematics \\ University of Utah \\ Salt Lake City, UT 84112}
\email{schwede@math.utah.edu}

\title{
The Brian\c{c}on-Skoda theorem for pseudo-rational and Du~Bois singularities and uniformity in excellent rings
}

\date{\today}
\allowdisplaybreaks
\begin{document}

\begin{abstract}
    Suppose $J = (f_1, \dots, f_n)$ is an $n$-generated ideal in any ring $R$.  We prove a general Brian\c{c}on-Skoda-type containment relating the integral closure $\overline{J^{n+k-1}}$ with ordinary powers $J^k$. 
    We prove that our result implies the full Brian\c{c}on-Skoda containment $\overline{J^{n+k-1}} \subseteq J^k$ for pseudo-rational singularities (for instance regular rings), and even for the weaker condition of birational derived splinters.  Our methods also yield the containment $\overline{J^{n+k}} \subseteq J^k$ for Du Bois singularities and even for a characteristic-free generalization.  Our Brian\c{c}on-Skoda-type theorem also implies well-known closure-based Brian\c{c}on-Skoda results $\overline{J^{n+k-1}} \subseteq (J^k)^{\mathrm{cl}}$ where, for instance, $\mathrm{cl}$ is tight or plus closure in characteristic $p > 0$, or $\mathrm{ep}$ closure or extension and contraction from $\widehat{R^+}$ in mixed characteristic.  Our proof relies on a study of the tensor product of the derived image of the structure sheaf of a partially normalized blowup of $J$ with the Buchsbaum-Eisenbud complex (equivalently the Eagon-Northcott complex) associated to $(f_1,\dots,f_n)^k$. 
    
As an application of our results and methods above, we prove the uniform Artin-Rees theorem and the uniform Brian\c{c}on-Skoda theorem for quasi-excellent, respectively quasi-excellent reduced, rings of finite dimension, answering conjectures of Huneke.   
\end{abstract}

\maketitle


\section{Introduction}

Suppose $R$ is a ring and $J = (f_1, \dots, f_n)$ is an ideal.
In \cite{BrianconSkoda} analytic methods were used to show the famous \myBrianconSkoda theorem, namely that 
\[
\overline{J^{n+k-1}} \subseteq J^k
\]
for all $k\in\mathbb{N}$ if $R$ is the coordinate ring of a smooth variety over $\mathbb{C}$ and 
where $\overline{(-)}$ denotes integral closure.  This was generalized to all regular rings in \cite{LipmanSathayeJacobianIdealsAndATheoremOfBS}.  Alternate proofs and improvements in various cases can be found for instance in  \cite{AberbachHuneke.ImprovedBS,Lipman.AdjointsInRegular,AberbachHunekeTrung.ReductionNumbersBSDepth,HyryVillamayor.ABSTheorem,EinLazarsfeld.AGeometricEffectiveNullstellensatz,Schoutens.NonstandardBS,Andersson.ExplicitVersionsOfBS, Sznajdman.ElementaryBSProof}.  There have been many such results for singular rings as well.  {For many Noetherian rings\footnote{This includes rings essentially of finite type over a field,  DVR, or $\bZ$, and $F$-finite rings of characteristic $p>0$.}, Huneke proved that there exists some integer $m$ such that $\overline{J^{m+k}} \subseteq J^k$ for \emph{all} ideals $J\subseteq R$, this is the so-called \emph{uniform \myBrianconSkoda} theorem \cite{Huneke.UniformBoundsInNoetherianRings} (\cf \cite{Huneke.DesingsAndUniformArtinRees,AnderssonSamuelssonSznajdman.OnBSForSingular,Sznajdman.BSForNonReduced,CidRuizJeffries-UniformityInNonReducedRings,KatzPolstra.UniformBrianconSkodaNonReduced}).  He then used this to prove a \emph{uniform Artin-Rees} theorem for these rings,
namely for each containment of finitely generated $R$-modules $N \subseteq M$, 
there exists an integer $\ell$ such that for \emph{all} ideals $I\subseteq R$, we have $I^kM \cap N \subseteq I^{k-\ell} M$ for all $k\geq \ell$ (\cf \cite{EisenbudHochster.NullstellensatzWithNilpotents,OCarroll.AUniformArtinRees,DuncanOCarroll.AFullUniformArtinRees,Wang.UniformPropertiesOf2Dimensional}). We will prove such a uniform Artin-Rees and uniform Brian\c{c}on-Skoda theorem for all excellent (resp. excellent reduced) rings of finite dimension, resolving Huneke's conjectures \cite[Conjecture 1.3 and Conjecture 1.4]{Huneke.UniformBoundsInNoetherianRings}.}  

On the more precise side, Lipman-Teissier \cite{LipmanTeissierPseudoRational} proved that $\overline{J^{\dim R+k-1}} \subseteq J^k$ if the ambient ring $R$ has pseudo-rational singularities (a characteristic free version of rational singularities); compared with the nonsingular case, one notes that $n$ was replaced\footnote{By enlarging the residue field, we can assume that $J$ is generated up to integral closure by at most $\dim R$ elements.} by $\dim R$, so this statement is not quite optimal. 
Inspired by \cite{LipmanTeissierPseudoRational}, Aberbach-Huneke \cite{AberbachHuneke.FRationalAndIntegralClosure} proved the full version of the \myBrianconSkoda theorem, namely that $\overline{J^{n+k-1}} \subseteq J^k$, for excellent $F$-rational rings (a characteristic $p > 0$ analog of rational singularities which are pseudo-rational by \cite{SmithFRatImpliesRat}).  This implies that $\overline{J^{n+k-1}} \subseteq J^k$ for rational singularities in characteristic zero via reduction mod $p \gg 0$ (\cite{HaraRatImpliesFRat,MehtaSrinivasRatImpliesFRat}), see also \cite{Li.BrianconSkodaAnalytic}.  However, it was unclear whether the full bound holds for pseudo-rational rings in positive or mixed characteristic.  { Below, in \hyperlink{CorollaryA}{{Corollary A}}, we obtain the full version of the \myBrianconSkoda theorem for generalizations of pseudo-rational singularities.}
For sufficiently singular rings, it is not always the case that $\overline{J^{n+k-1}} \subseteq J^k$.
However, Hochster-Huneke showed in \cite{HHmain} that 
\[
\overline{J^{n+k-1}} \subseteq (J^k)^*
\]
assuming $R$ is a Noetherian ring of characteristic $p > 0$ and where $(-)^*$ denotes the tight closure, a way to slightly enlarge the ideal -- a trivial operation if the ring has sufficiently mild singularities.  Related results can then be deduced in equal characteristic zero via reduction mod $p \gg 0$, see \cite{HochsterHunekeTightClosureInEqualCharactersticZero,Schoutens.NonStandardTightClosureForAffineCAlgebras,aschenbrennerschoutens,BrennerHowToRescueSolidClosure}.  Back in characteristic $p > 0$, the tight closure \myBrianconSkoda theorem was improved to the plus closure version in \cite{HochsterHunekeApplicationsofBigCM}.  Analogous closure related mixed characteristic results can be found in \cite{Heitmann.ThePlusClosureInMixed,heitmannepf,HeitmannMaExtendedPlusClosure,MSTWWAdjoint,RodriguezVillalobosSchwede.BrianconSkodaProperty} the last of which utilizes \cite{BhattAbsoluteIntegralClosure}.

Our first main theorem, {which has a remarkably simple proof}, implies, strengthens, and unifies many of the results above.

\hypertarget{MainTheorem}{
\begin{mainthm*}[\autoref{thm.DerivedMain}]
Suppose $R$ is a ring 
and $J = (f_1, \dots, f_n)$ is an ideal.  Set $Y \to \Spec R$ to be the blowup of $\overline{J^{n+k-1}}$ (or any map that dominates it, such as the normalized blowup of $J$ in $\Spec R$ if $R$ is reduced and has finitely many minimal primes\footnote{Note that normalizations are frequently defined only for schemes that locally have finitely many minimal primes.}).  Then $\overline{J^{n+k-1}}$ maps to zero in 
\[
 H_0\big( L^k(\underline{f}) \otimes^{\myL} \myR \Gamma(Y, \cO_{Y})\big)
\]
where $L^k(\underline{f})$ is the Buchsbaum-Eisenbud $L$-complex 
associated to $(f_1, \dots, f_n)^k$.  
\vskip 1pt
In fact, setting $J \cO_Y = \cO_Y(-E)$, we even have that the canonical map 
\[
    \myR\Gamma(Y, \cO_Y(-(n+k-1)E)) \to L^k(\underline{f}) \otimes^{\myL} \myR\Gamma(Y, \cO_Y)
\]
is zero in the derived category.
\end{mainthm*}}

The point is that if $r \in J^k$, or equivalently if $r \in R$ maps to zero in $R/J^k$, then $r$ also maps to zero in $H_0( L^k(\underline{f}) \otimes^{\myL} \myR \Gamma(Y, \cO_{Y}))$ as we discuss below.  But it turns out that mapping to zero in $H_0( L^k(\underline{f}) \otimes^{\myL} \myR \Gamma(Y, \cO_{Y}))$ is quite close to, and sometimes equivalent to, actually being in $J^k$.
{An interesting aspect of the proof of this result is that it does not use any vanishing theorems or Cohen-Macaulayness, unlike many proofs of \myBrianconSkoda-type results.} Let us briefly explain the complex $L^k(\underline{f})$.
If $f_1, \dots, f_n$ is a regular sequence on $R$, then $L^k(\underline{f})$ resolves $R/J^k$, that is, it is isomorphic to $\bZ[x_1, \dots, x_n]/(x_1, \dots, x_n)^k \otimes^{\myL}_{\bZ[x_1, \dots, x_n]} R$ in the derived category where $x_i \mapsto f_i$. In general, $L^k(\underline{f})$ is isomorphic to a specialization of the Eagon-Northcott complex.   {In particular, for $k = 1$, it is simply the Koszul complex.}  See \autoref{subsec.Buchsbaum-Eisenbud} and \cite{EagonNorthcott.IdealsDefinedByMatricesAndCertainComplex,BuchsbaumEisenbudFreeRes,Srinivasan.Thesis,SrinivasanDGA} for more details about this complex.  

A Noetherian reduced ring $R$ is said to be a \emph{birational derived splinter} if $R \to \myR \Gamma(Y, \cO_Y)$ splits for all proper birational maps $\pi : Y \to \Spec R$.  That is, if there exists $\phi : \myR \Gamma(Y, \cO_Y) \to R$ such that 
$
    R \to \myR \Gamma(Y, \cO_Y) \xrightarrow{\phi} R
$
is an isomorphism, see \cite{Lyu.PropertiesBirationalDerivedSplinters,BhattDerivedDirectSummandCharP,KovacsRat,DeDeynLankManali-RahulVenkatesh.MeasuingBirationalDerivedSplinters}.  We observe below in \autoref{lem.BirationalSplinterImpliesBirationallyPure} that every pseudo-rational local ring is a birational derived splinter via an argument of Kov\'acs (the converse holds in the Cohen-Macaulay analytically unramified case).  If $R$ is a birational derived splinter, we see that  
\[
    \ker\Big(R \to H_0(L^k({\underline f}) \otimes^{\myL} \myR \Gamma(Y, \cO_Y))\Big) \subseteq \ker\Big(R \to H_0(L^k({\underline f}) \otimes^{\myL} R)\Big) = J^k.
\]
As a consequence we obtain the following result.

\hypertarget{CorollaryA}{
\begin{corollaryA*}[{\autoref{thm.StrongLipmanTeissier}}]
    Suppose $R$ is a birational derived splinter (e.g., $R$ is  pseudo-rational, or $R$ is excellent and one of $F$-rational, BCM-rational or $+$-rational, see \autoref{rmk. FratBCMratPlusrat}), then 
    \[
        \overline{J^{n+k-1}} \subseteq J^k
    \]
    for any $n$-generated ideal $J$.  The result also holds for any ideal which locally has analytic spread $\leq n$.
\end{corollaryA*}
}
This recovers the main result of \cite{AberbachHuneke.FRationalAndIntegralClosure} 
and proves its analog in mixed characteristic. It provides simple proofs of the main results in \cite{LipmanSathayeJacobianIdealsAndATheoremOfBS,LipmanTeissierPseudoRational} and generalizes them to pseudo-rational singularities and beyond.   

More generally, suppose $R$ is a Noetherian integral domain and $Y \to \Spec R$ is an alteration with $Y$ integral and nonsingular.  After picking a map $\psi : \myR\Gamma(Y, \cO_Y) \to R$ such that $R \to \myR\Gamma(Y, \cO_Y) \xrightarrow{\psi} R$ is multiplication by $c \neq 0$ (which always exists), our method combined with the fact that $Y$ is a global birational derived splinter shows that $c \overline{J^{n+k-1}} \subseteq J^k$ independent of $J$. 
Having such a uniform $c$ was exactly the missing piece needed to prove the uniform Artin-Rees theorem for excellent rings of finite dimension or to prove the uniform \myBrianconSkoda theorem for excellent reduced rings of finite dimension.  While we do not know the existence of regular alterations in this level of generality, Gabber's weak local uniformization (\cite[Expos\'e VII, Theorem
1.1]{GabberWeakAlteration}) provides a local variant of regular alterations. Utilizing Gabber's result together with an intricate construction of nonsingular diagrams of schemes and derived splitting techniques, we obtain a $c\neq 0$ such that $c \overline{J^{n+k-1}} \subseteq J^k$ independent of $J$, see \autoref{thm.UniformBS}.  Then, combining Huneke's work \cite{Huneke.UniformBoundsInNoetherianRings} with \cite{ZhouUniformAnnihilators,LyuFormalLifting}, we obtain the following.
\begin{mainapp*}[{\autoref{cor.UniformAR}, \autoref{cor.UniformBS}, \cf \cite[Conjectures 1.3 and 1.4]{Huneke.UniformBoundsInNoetherianRings}}]
    Suppose $R$ is a quasi-excellent Noetherian ring of finite dimension and $N \subseteq M$ are finitely generated $R$-modules.  Then there exists an integer $\ell$ depending on $N$ and $M$ such that for all ideals $I\subseteq R$ and all $n \geq \ell$, we have that 
    \[
        I^n M \cap N \subseteq I^{n-\ell} M.
    \]
    If, in addition, $R$ is reduced, then there exists an integer $k$ depending only on $R$ so that for all ideals $I\subseteq R$ and all $n \geq k$, we have 
    $$\overline{I^n}\subseteq I^{n-k}.$$
\end{mainapp*}


Returning to our more precise \myBrianconSkoda theorem, a slight variation of our method also yields the following result.

\begin{corollaryB*}[\autoref{cor.DBSplintersHaveSkoda}]
If $R$ is Du Bois, or lim-perfectoid-pure (e.g., $F$-pure), or Cohen-Macaulay and lim-perfectoid-injective (e.g., Cohen-Macaulay and $F$-injective), then
\[
        \overline{J^{n+k}} \subseteq J^k
    \]
    for any $n$-generated ideal $J$.  The result also holds for any ideal which locally has analytic spread $\leq n$.
\end{corollaryB*}
In fact, our \autoref{cor.DBSplintersHaveSkoda} works for any {\it \DBSplinter}, a characteristic-free weakening of Du Bois singularities, so it generalizes \cite{HunekeWatanabe.UpperBoundOfMultiplicityOfFPure,WheelerZhang.RemarksOnBS} and extends them to mixed characteristic. In a related direction, we also obtain in \autoref{cor.PerfectoidBrianconSkoda} that 
$J_{\perfd}\overline{J^{n+k-1}} \subseteq J^k$ for an $n$-generated ideal $J$ in any perfectoid (so not Noetherian) ring.

Finally, our result also implies and unifies several ideal-closure versions of the \myBrianconSkoda theorem.  
Via the canonical map ${L^k(\underline{f}) \to R/J^{k}}$, we have
$
    \overline{J^{n+k-1}} \subseteq \ker \Big(R \to H_0( R/J^k \otimes^{\myL} \myR \Gamma(Y, \cO_Y)) \Big).$
If $R$ is Noetherian and reduced, consider the union over all proper birational maps $Y \to \Spec R$
\[
    (J^k)^{\Bir} := \bigcup_{Y \to \Spec R}\ker \Big(R \to H_0( R/J^k \otimes^{\myL} \myR \Gamma(Y, \cO_Y)) \Big)
\]
which we call the \emph{birational pre-closure} of $J^k$.  We view it as an analog of the $+$-closure of an ideal but in the birational direction.  
In positive (respectively mixed) characteristic, $R \to \myR\Gamma(Y, \cO_Y)$ factors the map $R \to R^+$ (respectively $R \to \widehat{R^+}$), see \cite{BhattDerivedDirectSummandCharP,BhattAbsoluteIntegralClosure,HochsterHuneke.InfiniteIntegralExtensions,HunekeLyubeznikAbsoluteIntegralExtensions}.  Thus we immediately obtain the following.
\begin{corollaryC*}[{\autoref{prop.CharPContainment}, \autoref{prop.MixedCharComparisonToClosure}}]
    Suppose $R$ is a Noetherian reduced ring with $p$ in its Jacobson radical and $I$ and $J = (f_1, \dots, f_n)$ are ideals. Then $I^{\Bir} \subseteq I \widehat{R^+} \cap R$ and so $\overline{J^{n+k-1}} \subseteq (J^k \widehat{R^+}) \cap R$, recovering \cite[Theorem 7.1]{HochsterHunekeApplicationsofBigCM} in characteristic $p$ and \cite[Corollary 5.1]{RodriguezVillalobosSchwede.BrianconSkodaProperty} in mixed characteristic. These then imply that $\overline{J^{n+k-1}} \subseteq (J^k)^*$ of \cite[Theorem 5.4]{HHmain} and that $\overline{J^{n+k-1}} \subseteq (J^k)^{\mathrm{ep}}$ of \cite[Theorem 4.2]{heitmannepf}.
\end{corollaryC*}

\subsection{History of this paper and strategy of the proof}

This paper came out of  \cite{EpsteinMcDonaldRGSchwede} as an attempt to find a \myBrianconSkoda theorem for the Hironaka-closure in characteristic zero.  The starting point was \cite{LipmanTeissierPseudoRational}, where Lipman-Teissier showed that if $h \in \overline{J^n}$ for $J = (f_1, \dots, f_n)$, then the \Cech class $[{h \over f_1 \cdots f_n}]$ is zero in $H^n_J(\myR\Gamma(Y, \cO_Y))$ where $Y$ is the normalized blowup of $J$.  As Koszul homology limits to local cohomology, we began by trying to show the stronger statement that $h \mapsto 0 \in H_0\big(\operatorname{Kos}(\underline{f}, \myR \Gamma(Y, \cO_Y))\big)$; indeed for small values of $n$ this is easy to see quite explicitly, and we were able to prove it in general.  Replacing the Koszul complex with a Buchsbaum-Eisenbud complex (or equivalently an Eagon-Northcott complex) we obtained our main result by a direct computation of the double complex $L^k(\underline{f}) \otimes \myR\Gamma(Y, \cO_Y)$.    
Finally, we discovered our less computational proof of the main theorem, also compare with \cite[Proof of Theorem B]{LazarsfeldLee.LocalSyzygiesOfMultiplierIdeals}. {The proof of the main application came after we realized that our result yielded a short proof that $\mathrm{T}(R) \neq 0$ assuming the existence of regular alterations. } 


\subsection*{Acknowledgements}
This paper grew out of studying the Hironaka closure in \cite{EpsteinMcDonaldRGSchwede} where Neil Epstein is also an author.  We deeply thank Neil Epstein for many valuable conversations on the topic of this paper.  The authors would also like to thank Rahul Ajit, Bhargav Bhatt, Rankeya Datta, Daniel Erman, Federico Galetto, Sean Howe, 
Craig Huneke, Annette Huber-Klawitter, Srikanth Iyengar, S\'andor Kov\'acs, Robert Lazarsfeld,    Zhenqian Li, Shiji Lyu, Wenbo Niu, Thomas Polstra, Claudiu Raicu, Ilya Smirnov, Shunsuke Takagi, and Keller VandeBogert
for valuable discussions related to this project and comments on earlier drafts.  {We particularly thank Keller VandeBogert for explaining the Buchsbaum-Eisenbud complex to us, and especially thank Bhargav Bhatt for several conversations related to hypercovers.}  Finally, we thank the referee for numerous useful comments and suggestions.

The first author received support from NSF grants DMS-2302430 and DMS-2424441.
The second author would like to acknowledge the support of the Pacific Institute for the Mathematical Sciences and a grant from the Simons Foundation International [SFI-MPS-T-Institutes-00020822, OY]. The third author received support from NSF grant DMS-2424326. The fourth author received support from NSF grants DMS-2101800, DMS-2501903 and Simons Foundation Travel Support for Mathematicians SFI-MPS-TSM00013051 while working on this project.  Work on this project was done by all the authors at the BIRS workshop ``Notions of Singularity in Different Characteristics'' in Banff Canada in October 2025.

\section{Integral closure and the Buchsbaum-Eisenbud complex}

In this section we show our main result.  We begin with some preliminaries on \Cech covers.

\subsection{Partially normalized blowups and \Cech covers}\label{subsec.BlowupsAndCech}

We begin by discussing partially normalized blowups of $J \subseteq R$.  For more information, we refer the reader to \cite[\href{https://stacks.math.columbia.edu/tag/052P}{Tag 052P}]{stacks-project} and \cite[\href{https://stacks.math.columbia.edu/tag/01OF}{Tag 01OF}]{stacks-project}.  For those most interested in the geometric case, what follows below becomes easier to contemplate if $R$ is a domain and so we invite the reader to think in that case if they prefer.  Suppose that $J = (f_1, \dots, f_n)$.  Then the blowup of $J$, $Y \to \Spec R$ is covered by affine charts corresponding to each $f_i$, $Y_i := \Spec R[{f_1 \over f_i}, \dots, {f_n \over f_i}]$ where here we view $R_i := R[{f_1 \over f_i}, \dots, {f_n \over f_i}]$ as a subring of $R[{1 \over f_i}]$ (in particular, if $g f_i = 0$, then $g$ becomes zero in this ring).  We remind the reader of the following.
\begin{enumerate}
    \item The blowups of $J$ and $J^m$ agree for any integer $m > 0$.  This follows either via the universal property of blowups or direct computation.  In particular, for the blowup of $J^m$, one only needs the charts $Y_i$ which correspond to $f_i^m \in J^m$.  
    \item Suppose $h \in \overline{J}$ with associated ideal-integral closure relation $h^d + a_1 h^{d-1} + \dots + a_d = 0$ with $a_i \in J^i$, then we see that $h/f_i$ satisfies the ring-integral closure relation 
    \[
        \left({h \over f_i}\right)^d + {a_1 \over f_i} \left({h \over f_i}\right)^{d-1} + \dots + {a_d \over f_i^d} \in R_i.
    \]
    Furthermore, if we blowup $(J, h)$, the chart corresponding to inverting $h$ is redundant. Indeed, in the Rees algebra $R[(J, h)t]$ (with $t$ a dummy variable) we see, by utilizing the original integral equation for $h$, that a homogeneous prime containing $ft$ for all $f \in J$ must contain $ht$.
    Hence we see that the blowup of $(J, h)$ is integral over the blowup of $J$. Taking a limit, the blowup of $\overline{J^{n+k-1}}$  has an integral map to the blowup of $J$.
\end{enumerate}

In what follows, we set $Y \to \Spec R$ to be the blowup of $\overline{J^{n+k-1}}$ with the affine cover whose charts $Y_i$ correspond to $f_i$ as above. Since $Y \to \Spec R$ factors through the blowup of $J$, we have that $J \cO_Y = \cO_Y(-E)$ where $E$ is an effective Cartier divisor ($E$ is not necessarily a Weil divisor in our generality as $Y$ may not be normal).  Observe also that $\overline{J^{n+k-1}} \cO_Y = \cO_Y(-(n+k-1)E)$.

{
}


\subsection{The Buchsbaum-Eisenbud complex}
\label{subsec.Buchsbaum-Eisenbud}
In this subsection, we provide some background on the Buchsbaum-Eisenbud complex, first introduced in \cite{BuchsbaumEisenbudFreeRes}.  It is isomorphic to an appropriate Eagon-Northcott complex \cite{EagonNorthcott.IdealsDefinedByMatricesAndCertainComplex}.  We largely follow the exposition in \cite[Section~1]{SrinivasanDGA}.  

Let $J=(f_1,\dots,f_n)$ be an ideal and let $F\to R$ be a map of free modules where $F$ has basis $e_1,\dots,e_n$ and $e_i\mapsto f_i$. The Buchsbaum-Eisenbud complex for $J^k$ is then the complex
\[0\to L_n^k(F)\to L_{n-1}^k(F)\to\dots\to L_2^k(F)\to L_1^k(F)\to R\to0\]
where $L_i^k(F)$ is the image of the natural map
\[\bigwedge^iF\otimes_RS_{k-1}F\to\bigwedge^{i-1}F\otimes_RS_kF.\]
It easily follows that the $0$th homology of $L^k(\underline{f})$ is $R/J^k$.  Indeed, if $f_1,\dots, f_n$ form a regular sequence, then  $L^k(\underline{f})$ is a free resolution of $R/J^k$.



We now globalize this perspective to the blowup of our ideal. Let $Y\to \Spec R$ be a (partially) normalized blowup of $J=(f_1,\dots,f_n)$ as above and recall that $J\cO_Y=\cO_Y(-E)$.  {It is standard to create an exact Koszul-type complex of locally free sheaves on such a blowup (see for instance \cite[Proof of 9.6.31]{LazarsfeldPositivity2}), we explain the generalization of this to the Buchsbaum-Eisenbud complex or Eagon-Northcott complex.}  Let $\{Y_i\}_{i=1}^n$ be the standard affine charts and consider $L^k(R_j)$, the Buchsbaum-Eisenbud complex associated to the $k$th power of the (unit) ideal $(\frac{f_1}{f_j},\dots,\frac{f_n}{f_j})\subseteq R_j$:
$$    0 \to R_j^{b_n} \xrightarrow{\partial_n} R_j^{b_{n-1}} \xrightarrow{\partial_{n-1}} \cdots \xrightarrow{\partial_3} R_j^{b_2} \xrightarrow{\partial_2} R_j^{b_1}\xrightarrow{\partial_1} R_j \to 0.$$
Note that $\partial_1\colon R_j^{b_1}\to R_j$ has its entries the degree $k$ monomials in $\frac{f_1}{f_j},\dots,\frac{f_n}{f_j}$ (which includes $\frac{f_j}{f_j}=1$) while the remaining differentials are matrices whose entries are $\pm\frac{f_i}{f_j}$. Furthermore, this complex is exact as its homologies are annihilated by the unit ideal. We next consider the following map of complexes: 
\[
\begin{tikzcd}
    0\arrow[r]& R_j^{b_n}\arrow[r,"\partial_n"]\arrow[d,"f_j^{-(n-1)}"]&R_j^{b_{n-1}}\arrow[r]\arrow[d,"f_j^{-(n-2)}"]&\cdots\arrow[r]&R_j^{b_2}\arrow[r,"\partial_2"]\arrow[d,"f_j^{-1}"]&R_j^{b_1}\arrow[r,"\partial_1"]\arrow[d,"1"]& R_j\arrow[d,"f_j^k"] \ar[r] & 0 \\
    0\arrow[r]& f_j^{-(n-1)}R_j^{b_n}\arrow[r,"f_j\partial_n"]&f_j^{-(n-2)}R_j^{b_{n-1}}\arrow[r]&\cdots\arrow[r]&f_j^{-1}R_j^{b_2}\arrow[r,"f_j\partial_2"]&R_j^{b_1}\arrow[r,"f_j^k\partial_1"]& f_j^kR_j \ar[r] & 0
\end{tikzcd}
\]
By design, this map is an isomorphism. Furthermore, the differentials of the bottom complex are simply matrices whose entries are $\pm f_i$ (except for $f_j^k\partial_1$ whose entries are the degree $k$ monomials in $f_1,\dots,f_n$) and thus its differentials are independent of $j$. Therefore we can glue these complexes on the charts $\{Y_i\}_{i=1}^n$ together to obtain the following exact complex:
\[
   0 \to \cO_Y((n-1)E)^{\oplus b_n}\to \cO_Y((n-2)E)^{\oplus b_{n-1}}\to\cdots\to\cO_Y^{\oplus b_1}\to \cO_Y(-kE) \to 0.
\]
Twisting by $\cO_Y(-(n-1)E)$, we obtain an exact complex
\[
\text{BE}_2\colon \hspace{0.5em}   0 \to \cO_Y^{\oplus b_n}\to \cO_Y(-E)^{\oplus b_{n-1}}\to\cdots\to\cO_Y(-(n-1)E)^{\oplus b_1}\to \cO_Y(-(n+k-1)E) \to 0
\]
which can be viewed as a subcomplex of the Buchsbaum-Eisenbud complex for $(f_1,\dots,f_n)^k$ on $Y$:
\[
\text{BE}_1\colon\hspace{0.5em}  0 \to \cO_Y^{\oplus b_n}\to \cO_Y^{\oplus b_{n-1}}\to\cdots\to\cO_Y^{\oplus b_1}\to \cO_Y\to 0.
\]

\begin{remark}
\label{rmk: Eagon-Northcott}
    Note that the Buchsbaum-Eisenbud complex is isomorphic to the so-called Eagon-Northcott complex by \cite[Cor.~3.2]{BuchsbaumEisenbudFreeRes}. We chose to work with the Buchsbaum-Eisenbud complex here, as it gives the obvious generating set for $(f_1,\dots,f_n)^k$ and is generally easier to describe than the Eagon-Northcott complex. Readers familiar with Eagon-Northcott complex can also check that there is a map from an exact Eagon-Northcott complex $\text{EN}_2$ on $Y$ to the pullback $\text{EN}_1$ of an Eagon-Northcott complex on $\Spec R$ to $Y$, where
    \begin{align*}
\text{EN}_2\colon \hspace{1em} & 0\to \bigwedge^{n+k-1}\cO_Y^{\oplus n+k-1} \otimes (\Sym^{n-1}\cO_Y(-E)^{\oplus k})^*\otimes \cO_Y(-(n-1)E) \to \cdots\to \\ 
& \bigwedge^{k+1}\cO_Y^{n+k-1}\otimes (\Sym^1\cO_Y(-E)^{\oplus k})^* \otimes \cO_Y(-(n-1)E) \to 
\bigwedge^{k}\cO_Y^{n+k-1} \otimes \cO_Y(-(n-1)E)  \\
& \to \bigwedge^k\cO_Y(-E)^{\oplus k} \otimes \cO_Y(-(n-1)E) \to 0
\end{align*} 
while
$$
\text{EN}_1\colon \hspace{0.5em} 0\to \bigwedge^{n+k-1}\cO_Y^{\oplus n+k-1} \otimes (\Sym^{n-1}\cO_Y^{\oplus k})^* \to \cdots\to \bigwedge^{k+1}\cO_Y^{n+k-1}\otimes (\Sym^1\cO_Y^{\oplus k})^* \to \bigwedge^{k}\cO_Y^{n+k-1} \to \bigwedge^k\cO_Y^{\oplus k} \to 0.
$$
These two complexes are both constructed from the $k\times(n+k-1)$ matrix:
\[
\begin{bmatrix}
    f_1 & f_2 & \dots & f_n & 0 & \dots  & \dots   & 0\\
    0 & f_1 & f_2 & \dots & f_n & 0 &\dots & 0 \\
    0 & 0 & f_1 & \dots & \dots   & f_n  & \dots, & 0\\
    \vdots & \vdots & \vdots & \ddots & \vdots  & \ddots & \ddots & \vdots \\
    0 & 0 & 0 & \dots  & \dots &  \dots & \dots & f_n
\end{bmatrix}.
\]
\end{remark}

\subsection{The main result}


In this subsection we prove our derived version of the \myBrianconSkoda theorem. The proof is remarkably simple and short.

\begin{theorem}
\label{thm.DerivedMain}
    Suppose $R$ is a ring and  $J = (f_1, \dots, f_n)\subseteq R$ is an ideal.  
    Suppose that $\pi : Y \to \Spec R$ is the blowup of $\overline{J^{n+k-1}}$ (or any map that dominates it) with $J\cO_Y=\cO_Y(-E)$ for $E$ an effective Cartier divisor.  Then the canonical map 
    \[
        \myR\Gamma(Y, \cO_Y(-(n+k-1)E)) \to L^k(\underline{f}) \otimes^{\myL} \myR\Gamma(Y, \cO_Y)
    \]
    is the zero map in the derived category.  In fact, even $\cO_Y(-(n+k-1)E) \to L^k(\underline{f}) \otimes \cO_Y$ is zero. In particular, by taking zeroth cohomology, the map 
    \[
        \overline{J^{n+k-1}} \to  H_0(L^k(\underline{f})\otimes^\bL \myR\Gamma(Y,\cO_Y))
    \]
    is zero.
\end{theorem}
\begin{proof}
By viewing $\cO_Y(-(n+k-1)E)$ as a complex in degree zero, we have a natural map $\cO_Y(-(n+k-1)E) \to \text{BE}_2$
and hence a natural composition map
\[
        \cO_Y(-(n+k-1)E)  \to \text{BE}_2 \to \text{BE}_1.\footnote{Alternatively, we can also consider the composition $\overline{J^{n+k-1}} \cO_Y \to \text{EN}_2 \to \text{EN}_1$, see \autoref{rmk: Eagon-Northcott}.}
    \]
Since $\text{BE}_2$ is exact and hence zero in the derived category, $\cO_Y(-(n+k-1)E) \to \text{BE}_1 = L^k(\underline{f})\otimes \cO_Y$ is zero in the derived category. Now simply apply the derived functor $\myR\Gamma(Y, -)$ and use the projection formula to obtain the conclusion. The final assertion follows by taking the zeroth cohomology and noticing that $\overline{J^{n+k-1}}\subseteq \Gamma(Y, \cO_Y(-(n+k-1)E))$. 
\end{proof}

\begin{remark}
There is another short argument explaining \autoref{thm.DerivedMain}. For any invertible ideal sheaf $\mathcal{I}\subseteq \cO_Y$ that is globally generated by $n$ elements $(f_1,\dots,f_n)$, the $\cO_Y$-module structure on the $\cO_Y$-algebra $\text{Kos}(f_1,\dots,f_n; \cO_Y)$ factors over $\cO_Y/\mathcal{I}^n$: it has $n$ cohomology groups (that sit in cohomological degrees $[-(n-1),0]$) and each of which is annihilated by $\mathcal{I}$ and thus $\mathcal{I}^{n}$ annihilates $\text{Kos}(f_1,\dots,f_n; \cO_Y)$ in $D(Y)$ by \cite[Lemma 3.2]{BhattDerivedDirectSummandCharP}. In general, by comparing with the exact complex $\text{BE}_2$ one sees that $H^{i}(\text{BE}_1)$ is annihilated by $\mathcal{I}$ for $-(n-1)\leq i\leq -1$ while $H^0(\text{BE}_1)$ is annihilated by $\mathcal{I}^k$. Thus $\mathcal{I}^{n+k-1}$ annihilates $\text{BE}_1$ in $D(Y)$ by \cite[Lemma 3.2]{BhattDerivedDirectSummandCharP}. We would like to thank Bhargav Bhatt for pointing this out to us.
\end{remark}

\begin{remark}
    The statement and proof of \autoref{thm.DerivedMain} also works verbatim if $\overline{J^{n+k-1}}$ is replaced, wherever it occurs, by any ideal $I$ such that $J^{n+k-1} \subseteq I \subseteq \overline{J^{n+k-1}}$.
\end{remark}




\section{Generalizations of rational and Du Bois singularities}

Suppose $\pi : Y \to X$ is a map of reduced Noetherian schemes.  We say $\pi$ is \emph{birational} if $\pi$ induces a bijection between the generic points of $X$ and $Y$ and furthermore induces an isomorphism between the residue fields at those generic points, see \cite[\href{https://stacks.math.columbia.edu/tag/01RO}{Tag 01RO}]{stacks-project}.
If $\pi$ is finite type, this is equivalent to the condition that $\pi$ induces a  bijection between the irreducible components of $Y$ and of $X$, and $\pi$ induces an isomorphism on a dense open set on each such irreducible component, see \cite[\href{https://stacks.math.columbia.edu/tag/0BAC}{Tag 0BAC}]{stacks-project}.

 We recall the definition of pseudo-rational rings, noting that regular rings are pseudo-rational \cite[Section 4]{LipmanTeissierPseudoRational}.
\begin{definition}[{\cite[Section 2]{LipmanTeissierPseudoRational}}]
\label{def.PseudoRational}
    Suppose $(R, \m)$ is a $d$-dimensional Noetherian local ring.  We say that $R$ has \emph{pseudo-rational singularities} if the following conditions hold.
    \begin{enumerate}
        \item $R$ is normal.
        \item $R$ is Cohen-Macaulay.
        \item The $\m$-adic completion $\widehat{R}$ is reduced (that is, $R$ is analytically unramified).
        \item For every proper birational map $\pi : Y \to \Spec R$ from a normal and integral $Y$, the induced  
        \[
            H^d_{\m}(R) \to H^d_{\m}(\myR \Gamma(Y, \cO_Y))
        \]
        is injective.  If $R$ has a dualizing module $\omega_R$, this just means that $\Gamma(Y, \omega_Y) \to \omega_R$ is surjective.\label{def.PseudoRational.RealCondition}
    \end{enumerate}
    We say that a Noetherian scheme is \emph{pseudo-rational} if its stalks are pseudo-rational.
\end{definition}

The earlier conditions of the definition imply that there are many normal $Y$ with $\pi : Y \to \Spec R$ proper birational.  In fact, for each proper birational $Y \to \Spec R$ with $Y$ integral, the normalization map $Y^{\mathrm{N}} \to Y$ is finite, see the discussion \cite[Pages 102, 103]{LipmanTeissierPseudoRational}.

Our next goal is to study a weakening of the pseudo-rational hypothesis.  For that we need a definition and some basic facts about pure maps in the derived category.

\begin{definition}[{\cite{Christensen.IdealsPhantomsGhostsSkeleta},\cite[Proposition 2.10]{BMPSTWW3}}]
    A map $M \to N$ in $D(R)$ with cone $Q$ is \emph{pure} if for each perfect complex $P$ with a map $P \to Q$, the composition $P \to M[1]$ is zero.  In other words, this means that $Q \to M[1]$ is phantom in the sense of \cite{Christensen.IdealsPhantomsGhostsSkeleta}, see \cite[Chapter 5]{Krause.HomologicalTheoryOfReps}.
    
    Equivalently, $M \to N$ is pure if $M \to N$ can be written as a filtered colimit of split maps $M \to N_i$ (in the derived $\infty$-category).
\end{definition}

\begin{lemma}[{\cf \cite[Proposition 3.10]{Lyu.PropertiesBirationalDerivedSplinters}}]
\label{lem.SplitVsPure}
    Suppose $R$ is a ring and $R \to S$ is faithfully flat.  Suppose $M \to N$ is a map in $D(R)$ such that the base change $M \otimes_R S=: M_S \to N_S := N \otimes_R S$ is pure.  Then $M \to N$ is also pure.  
\end{lemma}
\begin{proof}
    Let $Q$ denote the cone of $M \to N$ and $P \to Q$ is a map from a perfect complex to $Q$.  Our assumption implies that the base change $Q_S \to P_S \to M_S[1]$ is zero.  But $\Hom_{D(R)}(P, Q) \to \Hom_{D(S)}(P_S, Q_S)$ is injective since $P$ is perfect (\cite[\href{https://stacks.math.columbia.edu/tag/0A6A}{Tag 0A6A}]{stacks-project}) and so the result follows.
\end{proof}

   The following result generalizes the fact that a map of modules $f : M \to N$ with $N/f(M)$ finitely presented is pure if and only if it is split injective (\cite[Corollary 5.2]{HochsterRoberts.PurityOfFrobeniusAndLocalCohomology}, 
\cite[\href{https://stacks.math.columbia.edu/tag/058K}{Tag 058K}]{stacks-project}).

\begin{lemma}
\label{lem.LinquanBirationallyPureImpliesBirationalSplinter}
    Suppose $R$ is Noetherian, $M \to N \to C \xrightarrow{+1}$ is a triangle in $D(R)$ with $M\in D^+(R)$ and $C \in D^-_{\mathrm{coh}}(R)$.  Then $M \to N$ is pure if and only if $M \to N$ has a left inverse in $D(R)$.  
\end{lemma}
\begin{proof}
    Obviously if the map is split, then it is pure.  Srikanth Iyengar pointed out that the converse can be deduced from \cite[Proposition 5.2.8]{Krause.HomologicalTheoryOfReps} but we provide a proof for the reader's convenience.  First note that since $M\in D^+(R)$, there exists $n$ so that $\Hom_{D(R)}(K, M)=\Hom_{D(R)}(\tau^{\geq -n}K, M)$ for all $K\in D(R)$. Since $C \in D^-_{\mathrm{coh}}(R)$, 
    we can write $C$ as a colimit of perfect complexes $C_i$ so that $\tau^{\geq -n}C_i\simeq \tau^{\geq -n}C$ for all $i$. To show $M\to N$ is split in $D(R)$, we need to show that $\Ext^1(C, M)=H^1(\myR\lim\myR\Hom(C_i, M))=0$. By the Milnor exact sequence \cite[\href{https://stacks.math.columbia.edu/tag/0CQE}{Tag 0CQE}]{stacks-project}, we have 
    $$0\to {\lim}^1\Hom(C_i, M)\to H^1(\myR\lim\myR\Hom(C_i, M)) \to \lim\Ext^1(C_i, M) \to 0.$$
    The right term above vanishes because $M\to N$ is pure and $C_i$ is perfect, and the left term above also vanishes since $\Hom(C_i, M) = \Hom(\tau^{\geq-n}C_i, M)=\Hom(\tau^{\geq -n}C, M)$ and thus $\{\Hom(C_i, M)\}$ is a constant system.  Thus $\Ext^1(C, M)=0$ as we wanted.
\end{proof}

Lyu coined the following term, also see \cite{KovacsRat,Murayama.VanishingForQSchemes,DeDeynLankManali-RahulVenkatesh.MeasuingBirationalDerivedSplinters}.

\begin{definition}[{\cite{Lyu.PropertiesBirationalDerivedSplinters}}]
    A Noetherian reduced ring $R$ is a \emph{birational derived splinter} if for each proper birational  $\pi : Y \to \Spec R$, we have that $R \to \myR \Gamma(Y, \cO_Y)$ has a left inverse.  More generally, a Noetherian reduced scheme $X$ is a \emph{global birational derived splinter} if for each proper birational  $\pi : Y \to X$, we have that $\cO_X \to \myR \pi_* \cO_Y$ has a left inverse.    
\end{definition}

\begin{remark}
\label{rem.EnoughToConsiderBlowups}
    By Chow's lemma, to show that $R$ is a birational derived splinter, it suffices to consider projective and birational $\pi$, hence simply consider blowups. See  \cite[Chapter 8, Theorem 8.1.24]{Liu.AlgebraicGeometryAndArithmeticCurves}.  
\end{remark}


\begin{remark}
 Birational derived splinters are easily seen to be normal.  Indeed, if $R \to S$ is a finite birational extension, then $R \to S$ is not pure (equivalently split), unless $S = R$ (as the conductor of the map is the image of $\Hom_R(S, R) \to R$).  Hence, restricting birational derived splinters to domains is harmless since if $R$ has minimal primes $Q_1, \dots, Q_n$, then $R \to \prod R/Q_i$ is a finite extension.
\end{remark}

\begin{remark}
\label{rem.HenselizationPreservesBiratDerSplint}
    In what follows, it will often be helpful to replace our local ring $(R,\fm)$ with its strict Henselization $R^{\sh}$ in order to guarantee that $R$ has an infinite residue field. As such, we want to ensure that if $R$ is a birational derived splinter (or a reduced blowup-square splinter, see \autoref{def.BlowupSquareSplinter}) then so is $R^{\sh}$. This follows because proper birational maps to $\Spec R^{\sh}$ can be dominated by the base change of a proper birational map to $\Spec R$ by \cite[Theorem~5.1]{Lyu.PropertiesBirationalDerivedSplinters}. Alternatively, a similar reduction to the infinite residue field case can be obtained with $R[t]_{\fm[t]}$, see for instance \cite[Example~(c)~before~Theorem~2.1]{LipmanTeissierPseudoRational}.
\end{remark}


We believe the following is known to experts, indeed the argument can be found in the proof of \cite[Lemma 2]{KovacsRat}.  In the case that $X$ is regular, it also follows from \cite[Lemma 3.16]{LankVenkatesh.TriangulatedCategoriesOfSingularities}.

\begin{lemma}
\label{lem.BirationalSplinterImpliesBirationallyPure}
    Suppose $(R, \m)$ is a pseudo-rational local ring of dimension $d$, then $R$ is a birational derived splinter.  Even more, if $X$ is a pseudo-rational scheme with a canonical sheaf, then $X$ is a global birational derived splinter.
\end{lemma}
\begin{proof}
    Suppose $\pi : Y \to \Spec R$ is a proper birational map with $Y$ normal.  We work with the completion $\widehat{R}$ to gain access to a canonical module.  By \autoref{lem.SplitVsPure} and \autoref{lem.LinquanBirationallyPureImpliesBirationalSplinter}, it suffices to show  
    $
        \widehat{R} \to \widehat{R} \otimes^{\myL} \myR\Gamma(Y, \cO_Y) \cong \myR \Gamma(Y_{\widehat{R}}, \cO_{Y_{\widehat{R}}})
    $    
    is pure where $Y_{\widehat{R}} := Y \times_{\Spec R} \Spec \widehat{R}$.  We will show that $ \widehat{R} \to \myR \Gamma(Y_{\widehat{R}}, \cO_{Y_{\widehat{R}}})$ is split.  
    
    Now, $\widehat{R}$ is Cohen-Macaulay and has a normalized dualizing complex $\omega_{\widehat{R}}[d]$.  Since $H^d_{\m}(R) \to H^d_{\m}(\myR \Gamma(Y, \cO_Y))$ injects, we see that the Matlis dual 
    $
        \Gamma(Y_{\widehat{R}}, \omega_{Y_{\widehat{R}}}) \to \omega_{\widehat{R}}
    $
    surjects and so is an isomorphism (both sides are rank-1 torsion-free).  This yields a composition 
    \[
        \omega_{\widehat{R}}[d] \xrightarrow{\sim} \Gamma(Y_{\widehat{R}}, \omega_{Y_{\widehat{R}}})[d] \to \myR \Gamma(Y_{\widehat{R}}, \omega_{Y_{\widehat{R}}}^{\mydot}) \to \omega_{\widehat{R}}^{\mydot} \cong \omega_{\widehat{R}}[d]
    \]
    that one sees is an isomorphism (it is a map between S2-modules that is an isomorphism outside a set of codimension $\geq 2$).  Applying Grothendieck duality produces a composition
    \[
        \widehat{R} \leftarrow \myR\Gamma(Y_{\widehat{R}}, \cO_{Y_{\widehat{R}}}) \leftarrow \widehat{R}
    \]
    that is also an isomorphism.  This proves that $\widehat{R} \to \myR \Gamma(Y_{\widehat{R}}, \cO_{Y_{\widehat{R}}})$ is split, as desired.  
    
    For the scheme case, since one may assume $X$ is integral, we may assume that $\omega_X^{\mydot} = \omega_X[0]$ is a dualizing complex.  Set $\omega_Y^{\mydot} := \pi^! \omega_X$.  As above, we have a composition
    \[
        \omega_X^{\mydot} \xrightarrow{\sim} \pi_* \omega_{Y} \to \myR\pi_* \omega_Y^{\mydot} \to \omega_{X}^{\mydot}
    \]
    which is the identity.  Then dualize and argue as above.
\end{proof}

\begin{remark}
\label{rmk. FratBCMratPlusrat}
    Excellent local rings that are $F$-rational or BCM-rational (or even $+$-rational: i.e., $(R,\m,k)$ is a Cohen-Macaulay local domain with $\mathrm{char}\,k > 0$, such that $H^{\dim R}_{\m}(R) \to H^{\dim R}_{\m}(S)$ injects for all finite $R \subseteq S$) are birational derived splinters, as all are pseudo-rational \cite{SmithFRatImpliesRat,maschwede:2021}.  
    Regular rings are derived splinters in general \cite{Bhatt.DirectSummandAndDerivedVariant}, but the argument above that they are \emph{birational} derived splinters is much simpler than the fact that finite maps $R \to S$ split in mixed characteristic \cite{Andre.ConjectureDuFacteurDirect}.  
\end{remark}

\begin{remark}
    We note that excellent local Cohen-Macaulay birational derived splinters are pseudo-rational as condition \autoref{def.PseudoRational.RealCondition} of \autoref{def.PseudoRational} is clear from the splinter condition. However, without the Cohen-Macaulay condition, this is not true.

    Suppose $R \subseteq S$ is a pure inclusion of Noetherian domains and $S$ is a birational derived splinter, then it is not difficult to see that $R$ is a birational derived splinter as well, see  \cite[Corollary 3.12]{Lyu.PropertiesBirationalDerivedSplinters}. 
    But now by \cite[Example 8.6]{MaPolstra.FBook} (which utilizes the examples of \cite{Kovacs.NonCMCanonicalSings}), there exists such a split extension $R \subseteq S$ so that $S$ is $F$-rational and hence pseudo-rational but $R$ is not Cohen-Macaulay.  In particular, in this example, $R$ is a birational derived splinter but is not pseudo-rational.  
    
\end{remark}

\subsection{\myBrianconSkoda theorem for birational derived splinters}
Our discussion so far yields the following result.
\begin{theorem}
\label{thm.StrongLipmanTeissier}
    Suppose that a Noetherian ring $R$ is a birational derived splinter (for instance if $R$ is pseudo-rational, or excellent and one of $F$-rational, BCM-rational or $+$-rational).  Suppose that $J \subseteq R$ is an ideal that locally has analytic spread $\leq n$.  Then for every integer $k > 0$
    \[
        \overline{J^{n+k-1}} \subseteq J^k.
    \]
\end{theorem}
\begin{proof}
    We may assume that $R$ is a normal local domain.  Let $Y$ denote the blowup of $\overline{J^{n+k-1}}$.  We may also assume that $R$ is strictly Henselian (see \cite[Proposition 1.6.2]{SwansonHuneke.integralclosure} and \autoref{rem.HenselizationPreservesBiratDerSplint}, and note that blowups commute with flat base change) and so assume $R$ has infinite residue field.

    By \cite[Proposition 8.3.7]{SwansonHuneke.integralclosure}, $J$ has a minimal reduction $J' = (f_1, \dots, f_n) \subseteq J$ whose powers have the same integral closures as the powers of $J$.  Hence we may replace $J$ by $J'$ to assume that $J$ is $n$-generated. 
    Now by \autoref{thm.DerivedMain} and using that $R \to \myR\Gamma(Y, \cO_Y)$ splits, we have that $\overline{J^{n+k-1}}$ maps to zero in $H_0(L^k(\underline{f})) \otimes R = R/J^k$, i.e., $\overline{J^{n+k-1}} \subseteq J^k$. 
\end{proof}

\begin{remark}
\label{rem.SplittingByMultByC}
    Suppose that $Y \to \Spec R$ is the blowup of $\overline{J^{n+k-1}}$ and that there exists some map $\phi : \myR\Gamma(Y, \cO_Y) \to R$ such that composition $R \to \myR\Gamma(Y, \cO_Y) \xrightarrow{\phi} R$ is multiplication by $c$ (see \cite[Definition 2.21]{EpsteinMcDonaldRGSchwede} and \cite[Section 2.4]{Lyu.PropertiesBirationalDerivedSplinters}).  The proof above shows that  
    \[
        c \overline{J^{n+k-1}} \subseteq J^k.
    \]
    {For more discussion in this direction, see \autoref{sec.UniformBS}.}
    %
\end{remark}

{
\begin{remark}
\label{rem.LisResultNeedsCM}
Suppose $R$ is a normal domain essentially of finite type over $\C$.  If $R$ has rational singularities, then we know from \cite{AberbachHuneke.FRationalAndIntegralClosure} that the \myBrianconSkoda theorem holds.  In \cite{Li.BrianconSkodaAnalytic}, the author gave another proof of this fact   in the analytic setting, and also showed that $
\overline{\mathcal{J}^{n+k-1}}\cdot\omega_R\subset\mathcal{J}^k\cdot\omega_R$ 
assuming that $\omega_R = \Gamma(X, \omega_X)$ 
for $X \to \Spec R$ a resolution of singularities (roughly, $R$ is pseudo-rational without the Cohen-Macaulay hypothesis).
It was also claimed that one only needed this weakening of rational singularities without the Cohen-Macaulay condition to deduce that $\overline{J^{n+k-1}} \subseteq J^k$ generally.  

We believe this final statement is false in the algebraic category. 
Explicitly, let $E$ denote the elliptic curve defined by $x^3+y^3+z^3 = 0$ in $\bP^2_{\bC}$ and set $X = E \times \bP^1_{\bC}$ Segre embedded in $\bP^5_{\bC}$.  Set $R$ to be the affine cone over $X$.  
It is well-known that $R$ is not Cohen-Macaulay (use the K\"unneth formula to see that $H^1(X, \cO_X)\neq 0$).  
We explain why $\omega_R = \myR\Gamma(Y, \omega_Y)$ where $Y \to \Spec R$ is the blowup of the origin.  By \cite[Theorem 4.1]{Urbinati.DiscrepanciesOfNonQGor} and \cite[Theorem 8.2]{DeFernexHaconSingsOnNormal} we see that $\omega_R \cdot \cO_Y \subseteq \omega_Y$.  Hence we have a sequence of maps
\[
    \omega_R \to \Gamma(Y, \omega_R \cdot \cO_Y ) \to \Gamma(Y, \omega_Y) \to \omega_R.
\]
But these are nonzero maps of rank-1 torsion-free modules and hence they are injective.  Furthermore, these maps are the identity outside of the origin and hence the composition is an isomorphism as $\omega_R$ is reflexive.  It follows that $\Gamma(Y, \omega_Y) \to \omega_R$ is an isomorphism.

Now, it is well-known that $R$ is the Segre product $\C[x,y,z]/(x^3+y^3+z^3) \# \C[u,v]$. Let $J = ((xu)^2, (zu)^2) \subseteq R$ and $J' = ((xu)^2, (xu)(zu),(zu)^2) \subseteq \overline{J}$. Then we have $h := (xu)(yu)^2(zu) \in \overline{J'^2}=\overline{J^2}$  (use $h^3 \in J'^6$ for the first containment). On the other hand, if $h\in J$, then after specializing $u=v=1$, we would have $xy^2z\in (x^2,z^2)S$ where $S=\C[x,y,z]/(x^3+y^3+z^3)$. But since $x, z$ is a regular sequence in $S$, this would imply $y^2\in (x,z)S$ which is a contradiction. Thus $h\notin J$.  We found this example with the help of Macaulay2 \cite{M2}. 

\end{remark}
}










%

\subsection{A characteristic free weakening of Du Bois singularities}

Suppose $R$ is reduced and $Y \to \Spec R$ is the blowup of an ideal $J \subseteq R$.  Consider the reduced blowup square
\[
    \xymatrix
    {F \ar[r]\ar[d] & Y \ar[d]\\
    V(\sqrt{J}) \ar[r] & \Spec R}
\]
with $F$ the reduced preimage of $V(\sqrt{J})$.
We have the following triangle defining $D_J^\mydot$:
\begin{equation}
\label{eq.SplinterObjectForBlowupSquare}
    D_{J}^{\mydot} \to \myR\Gamma(Y, \cO_Y) \oplus R/\sqrt{J} \to \myR\Gamma(F, \cO_F) \xrightarrow{+1}.
\end{equation}


\begin{definition}
\label{def.BlowupSquareSplinter}
    Suppose $R$ is a Noetherian reduced ring.
    We say that $R$ is a \emph{\DBSplinter} if for each ideal $J\subseteq R$ and $D_{J}^{\mydot}$ as above, the induced map 
    $R \to D_{J}^{\mydot}$
    has a left inverse.
\end{definition}

Suppose $R$ is essentially of finite type\footnote{This can be relaxed in view of \cite{MurayamaInjectivityAndCubicalDescentForSchemesStacksSpaces} but we restrict to this setting.} over a field of characteristic zero.  Then, for instance by \cite{DuBoisMain,GNPP}, we have functorial maps $R \to \DuBois{R}, \cO_Y \to \DuBois{Y}, \cO_F \to \DuBois{F}, $ and $R/\sqrt{J} \to \DuBois{R/\sqrt{J}}$ where $\DuBois{}$ is the $0$th-graded piece of the Deligne-Du Bois complex. Hence from the triangle (see \cite[Proposition 3.9]{DuBoisMain})
$
    \DuBois{R} \to \myR\Gamma(Y, \DuBois{Y}) \oplus \DuBois{R/\sqrt{J}} \to \myR\Gamma(Y, \DuBois{F}) \xrightarrow{+1}
$
 we have a factorization 
\begin{equation}
\label{eq.DuBoisFactorization}
    R \to D_J^{\mydot} \to \DuBois{R}.
\end{equation}
This factorization also follows from hyperresolution constructions \cite{GNPP} or from \cite{HuberJorder.DifferentialFormInHTopology}.

\begin{proposition}
\label{prop.ThingsImplyDBSplinter}
The following singularities are \DBSplinters.
\begin{enumerate}
    \item $R$ is essentially of finite type over a field of characteristic zero and is Du Bois. \label{prop.ThingsImplyDBSplinter.a}
    \item $R$ has $p$ in its Jacobson radical and is lim-perfectoid pure (e.g., $R$ is $F$-pure).
    \item $R$ has $p$ in its Jacobson radical and is Cohen-Macaulay and lim-perfectoid injective (e.g., $R$ is Cohen-Macaulay and $F$-injective).
\end{enumerate}
\end{proposition}

\begin{proof}
For \autoref{prop.ThingsImplyDBSplinter.a}, recall that $R$ having Du Bois singularities means exactly that $R \to \DuBois{R}$ is an isomorphism.  Thus \autoref{eq.DuBoisFactorization} gives a left inverse of $R\to D_J^\mydot$.

For lim-perfectoid pure singularities we have a factorization 
\[
    R \to D_{J}^{\mydot} \to R_{\perfd},
\]
see for instance \cite[Lemma 5.3]{BMPSTWW3}. Thus by the definition of lim-perfectoid purity, $R\to R_{\perfd}$ is pure and thus $R\to D_{J}^{\mydot}$ is pure in $D(R)$. Hence the latter is split in $D(R)$ by \autoref{lem.LinquanBirationallyPureImpliesBirationalSplinter} (note that in characteristic $p$, lim-perfectoid purity is precisely $F$-purity by \cite[Remark 4.3]{BMPSTWW3}). 

Now we assume $R$ is Cohen-Macaulay and lim-perfectoid injective (by \cite[Remark 4.3]{BMPSTWW3}, in characteristic $p$ this is equivalent to $F$-injective). By considering the factorization $R \to D_{J}^{\mydot} \to R_{\perfd}$ as above and the definition of lim-perfectoid injective, we have that 
$
    H^d_{\m}(R) \to H^d_{\m}(D_{J}^{\mydot})
$
is injective for each maximal ideal $\m$.  After base changing to the $\m$-adic completion $\widehat{R}$, we have that 
$
    H^d_{\m}(\widehat{R}) \to H^d_{\m}(D_{J}^{\mydot} \otimes \widehat{R})
$
is injective.  Hence by local duality, we have a surjection
\[
    H^{-d}({\bf D}(D_{J}^{\mydot} \otimes \widehat{R})) \to \omega_{\widehat{R}}
\]
where ${\bf D}(-) = \myR\Hom_{\widehat{R}}(-, \omega_{\widehat{R}}^{\mydot})$ is Grothendieck duality and $\omega_{\widehat{R}}^{\mydot}$ is a normalized dualizing complex. 
\begin{claim}
The map $H^{-d}({\bf D}(D_{J}^{\mydot} \otimes \widehat{R})) \to \omega_{\widehat{R}}$ is an injection.
\end{claim}
\begin{proof}[Proof of claim]
The strategy is the same as \cite[Proposition 4.10]{KovacsSchwedeSmithLCImpliesDuBois}.  If we consider the triangle 
$
    R \to D_{J}^{\mydot} \to G^{\mydot} \xrightarrow{+1}
$
it suffices to show that $H^d_{\m}(G^{\mydot}) = 0$.  From the spectral sequence $H^p_{\m}(H^q(G^{\mydot})) \Rightarrow H^{p+q}_{\m}(G^{\mydot})$, it suffices to show that $\dim \Supp H^i(G^{\mydot}) \leq d - i - 1$ for every $i$.  

Now, as $R$ is weakly normal (see \cite{DattaMurayama.PermanenceOfFInjectivity,BMPSTWW3}), by taking the zeroth cohomology of the sequence \autoref{eq.SplinterObjectForBlowupSquare}, we see that $H^0(D_{J}^{\mydot}) = R$.  In particular, $H^i(G^{\mydot}) = H^i(D_{J}^{\mydot})$ for $i > 0$ and $H^0(G^{\mydot}) = 0$.    From \autoref{eq.SplinterObjectForBlowupSquare} we see that there is a triangle 
\[
    \myR\Gamma(Y, I_F) \to D_{J}^{\mydot} \to R/\sqrt{J} \xrightarrow{+1}.
\]
It follows that $H^i(D_{J}^{\mydot}) \cong H^i(Y, I_F)$ for $i > 0$.  Finally, $\dim \Supp H^i(Y, I_F) \leq d-i-1$ as $Y \to \Spec R$ has fiber dimension $\leq \dim(R_Q) -1$ at every localization $R_Q$ of $R$.
This proves the claim.
\end{proof}
Finally, we have a composition which is an isomorphism as in the pseudo-rational case:
\[
    \omega_{\widehat{R}}[d] \cong H^{-d}({\bf D} (D_{J}^{\mydot} \otimes \widehat{R}))[d] \to {\bf D} (D_{J}^{\mydot} \otimes \widehat{R}) \to \omega_{\widehat{R}}[d].
\]
 Dualizing proves that $\widehat{R} \to D_{J}^{\mydot} \otimes \widehat{R}$
splits for each maximal ideal $\m$. Using \autoref{lem.SplitVsPure},  \autoref{lem.LinquanBirationallyPureImpliesBirationalSplinter}, and \cite[\href{https://stacks.math.columbia.edu/tag/0D0C}{Tag 0D0C}]{stacks-project}, similar to \cite[2.4, Proposition 3.5]{Lyu.PropertiesBirationalDerivedSplinters}, completes the proof.
\end{proof}

Our next result generalizes \cite[Theorem 3.2(2)]{HunekeWatanabe.UpperBoundOfMultiplicityOfFPure} and several results in \cite{WheelerZhang.RemarksOnBS} and answers some questions therein.


\begin{theorem}
\label{cor.DBSplintersHaveSkoda}
    Suppose $R$ is a reduced ring and $J = (f_1, \dots, f_n) \subseteq R$ is an ideal.
    Let $\pi : Y \to \Spec R$ denote the blowup of $I := \overline{J^{n+k-1}}$ with $J \cO_Y = \cO_Y(-E)$ so that $I \cO_Y = \cO_Y(-(n+k-1)E)$.
    Then the canonical map
    \[
        \myR \Gamma\big(Y, \cO_Y(-(n+k)E)\big) \to L^k(\underline{f}) \otimes^{\myL}_R D_{I}^{\mydot}
    \]
    is zero in the derived category. Hence if $R$ is a Noetherian \DBSplinter (e.g., Du Bois, lim-perfectoid pure, or Cohen-Macaulay and lim-perfectoid injective), then for every integer $k> 0$,
    \[
        \overline{J^{n+k}} \subseteq J^k.
    \]
    The same holds if $J \subseteq R$ is an ideal that locally has analytic spread $\leq n$.  
    
    Furthermore, if $R$ is additionally  normal and Nagata (e.g., excellent), one obtains that 
    \[  
        J_{>n+k-1} \subseteq J^k
    \]
    where the left side is the fractional power integral closure as in \cite[Definition 10.5.3]{SwansonHuneke.integralclosure}.
\end{theorem}
\begin{proof}
    Set $F = E_{\reduced}$ with associated ideal sheaf $I_F$, and note it may not be a divisor since $Y$ need not be normal or even Noetherian.
   Since
    $
        \cO_Y(-(n+k-1)E) \to L^k(\underline{f}) \otimes \cO_Y
    $
    is zero in $D(Y)$ by \autoref{thm.DerivedMain}, so is the map
    $$
        \myR\Gamma(Y,\cO_Y(-(n+k)E)) \to L^k(\underline{f}) \otimes^{\myL}_R \myR\Gamma(Y,\cO_Y(-E)).
    $$
    We have a canonical map $\myR\Gamma(Y,\cO_Y(-E)) \to \myR\Gamma(Y, I_F) \to D_{I}^{\mydot}$ and now we see that the composition 
    \[
        \myR\Gamma(Y,\cO_Y(-(n+k)E)) \to D_{I}^{\mydot} \to L^k(\underline{f}) \otimes^{\myL}_R D_{I}^{\mydot}
    \]
    is zero in the derived category proving the first statement.  For the second statement, use the splitting $D_{I}^{\mydot} \to R$ and take zeroth cohomology.  
    The analytic spread statement follows as in the proof of \autoref{thm.StrongLipmanTeissier} as blowups, and hence $D_I^{\mydot}$, commute with base change to the strict Henselization.

    For the final statement in the normal Nagata case, replace $Y$ by the blowup of $\overline{J^m}$ for $m \gg 0$, which dominates the original $Y$ and is also normal. Note that $F$ is a reduced Weil divisor and we have
    \[
        \cO_Y(-(n+k-1)E - F) \to L^k(\underline{f}) \otimes_R \cO_Y(-F)
    \]
    is zero in $D(Y)$ (note that while $\cO_Y(-F)$ may not be locally free, every other term is). Arguing as above we obtain that the composition 
    $$\myR\Gamma(Y,\cO_Y(-(n+k-1)E - F)) \to D_{\overline{J^m}}^{\mydot} \to L^k(\underline{f}) \otimes^{\myL}_R D_{\overline{J^m}}^{\mydot}$$ is zero in the derived category and hence we see that $\Gamma(Y, \cO_Y(-(n+k-1)E - F))  \subseteq J^k$.  But the left side is $J_{>n+k-1}$ as $R$ and $Y$ are normal.
\end{proof}

The following immediate corollary recovers and generalizes the multiplicity bound of rational, $F$-rational, $F$-pure and Du Bois singularities established in \cite{HunekeWatanabe.UpperBoundOfMultiplicityOfFPure,ShibataMultiplicityDuBois,ParkMultiplicityDuBois}.

\begin{corollary} Let $(R,\m)$ be a Noetherian reduced local ring of dimension $d$ and embedding dimension $e$.
\begin{enumerate}
    \item If $R$ is a birational derived splinter (e.g., pseudo-rational), then $e(R)\leq\binom{e-1}{d-1}$.
    \item If $R$ is a \DBSplinter (e.g., Du Bois, lim-perfectoid pure, or Cohen-Macaulay and lim-perfectoid injective), then $e(R)\leq \binom{e}{d}$.
\end{enumerate}
\end{corollary}
\begin{proof}
The proof is essentially the same as that of \cite[Theorem 3.1]{HunekeWatanabe.UpperBoundOfMultiplicityOfFPure}.
Without loss of generality we may assume $R/\m$ is infinite (see \autoref{rem.HenselizationPreservesBiratDerSplint}). Let $J:=(x_1,\dots, x_d)$ be a minimal reduction of $\m$ and we extend it to a minimal generating set $(x_1,\dots,x_d,x_{d+1},\dots,x_e)$ of $\m$. By \autoref{thm.StrongLipmanTeissier} and \autoref{cor.DBSplintersHaveSkoda}, we know that $\m^d\subseteq J$ in case (a) and $\m^{d+1}\subseteq J$ in case (b). It follows that $R/J$ is generated by monomials in $x_{d+1},\dots,x_e$ of degree $\leq d-1$ in case (a) and $\leq d$ in case (b). In particular, we have 
\[e(R)=e(J)\leq \ell(R/J)\leq \begin{cases}
    \binom{e-1}{d-1} \text{ in case (a),} \\ 
    \binom{e}{d} \text{ in case (b)}.
\end{cases} \qedhere \]
\end{proof}

\subsection{\myBrianconSkoda theorem for perfectoid rings}
Since our \autoref{thm.DerivedMain} does not require $R$ to be Noetherian, it has the following interesting consequence on integral closure of ideals in perfectoid rings. We refer the reader to \cite{BhattScholzepPrismaticCohomology} for unexplained terminologies in what follows.
\begin{corollary}
\label{cor.PerfectoidBrianconSkoda}
Let $R$ be a perfectoid ring and $J=(f_1,\dots,f_n) \subseteq  R$. Then for every $k\geq 1$, we have
$$J_{\perfd}\overline{J^{n+k-1}} \subseteq J^k.$$ Furthermore, for every $k\geq 1$, we have 
$$\overline{J^{n+k}}\subseteq J^k.$$ 
\end{corollary}
\begin{proof}
Let $Y\to \Spec(R)$ be the blowup of $\overline{J^{n+k-1}}$. By \autoref{thm.DerivedMain}, $\overline{J^{n+k-1}}$ maps to zero in 
\[ H_0\big(L^k(\underline{f})\otimes_R^\bL \myR\Gamma(Y,\cO_Y)\big).\]
As $R$ is perfectoid, $Y^{\text{pfd}}\to \Spec(R)^{\text{pfd}}=\Spec(R)$ is a $J$-almost isomorphism (\cite[Corollary 8.12]{BhattScholzepPrismaticCohomology}). Hence $R\to \myR\Gamma(Y,\cO_Y)$ is $J$-almost split (since it factors the map to $\myR\Gamma(Y^{\text{pfd}},\cO_{Y, \text{perfd}})$ which is $J$-almost isomorphic to $R$). Therefore $\overline{J^{n+k-1}}$ has $J$-almost zero image in $H_0(L^k(\underline{f}))=R/J^k$ as wanted. 

For the second statement, note that since perfectoidization is an arc-sheaf, we know that $R$ is a direct summand of $D^\mydot_{\overline{J^{n+k-1}}}$ (see \cite[Lemma 5.3]{BMPSTWW3}). By \autoref{cor.DBSplintersHaveSkoda} (whose first statement does not assume $R$ is Noetherian), 
 \[
        \myR \Gamma\big(Y, \cO_Y(-(n+k)E)\big) \to L^k(\underline{f}) \otimes^{\myL}_R D^\mydot_{\overline{J^{n+k-1}}}
    \]
is zero in $D(R)$. Now use the splitting $D^\mydot_{\overline{J^{n+k-1}}}\to R$ and take zeroth cohomology to obtain that $\overline{J^{n+k}}\subseteq J^k$.
\end{proof}

\section{Connections with closure operations}

Inspired by \cite{EpsteinMcDonaldRGSchwede} we make the following definition.

\begin{definition}
Suppose that $R$ is a reduced Noetherian ring and $J \subseteq R$ is an ideal.  We define the \emph{birational pre-closure} of $J$, denoted $J^{\Bir}$, to be the union 
\[
    \bigcup_Y \ker\Big(R \to H_0\big(R/J \otimes^{\myL} \myR\Gamma(Y, \cO_Y)\big) \Big)
\]
where $Y \to \Spec R$ varies over all proper birational maps.    As the set of proper birational morphisms is filtered, we see that this union is in fact an ideal.  
\end{definition}

\begin{remark}
For general non-Noetherian $R$, we expect that one should replace the birational condition of the birational pre-closure with the $M$-morphisms of \cite{Lyu.PropertiesBirationalDerivedSplinters}. 
\end{remark}

Since $R$ is Noetherian $J^{\Bir}$ is computed by a single $Y$ or, by Chow's lemma, a single blowup.  Furthermore, the formation of $J^{\Bir}$ commutes with localization since blowups can be spread out.
If $R \to S$ is finite birational, we may take our $Y$ factoring through $\Spec S \to \Spec R$.  Hence we obtain:

\begin{lemma}
    Suppose $R$ is a reduced Noetherian ring, and $R \hookrightarrow S$ is a finite extension such that the induced $\Spec S \to \Spec R$ is birational (for example, if $S = \prod_i R/Q_i$ where the $Q_i$ are the minimal primes of $R$).  Then 
    $
        J^{\Bir} = (JS)^{\Bir} \cap R.
    $
\end{lemma}


We note the following immediate consequence of the definition of birational derived splinters.
\begin{lemma}
    If $R$ is a birational derived splinter, then $J^{\Bir} = J$ for every ideal $J \subseteq R$.
\end{lemma}

Given a resolution of singularities, $J^{\Bir}$ is computed by that resolution.

\begin{lemma}
\label{lem.BirClosureContainment}
    Suppose $\pi : X \to \Spec R$ is a proper birational map with $X$ a global birational derived splinter.    If $J \subseteq R$ is an ideal, then 
    \[
        J^{\Bir} = \ker\Big(R \to H_0(R/J \otimes^{\myL} \myR\Gamma(X, \cO_X)) \Big).
    \]    
\end{lemma}
\begin{proof}
For the main statement, it suffices to consider proper birational maps $Y \xrightarrow{\rho} X \xrightarrow{\pi} \Spec R$, but $\cO_X \to \myR\rho_* \cO_Y$ splits.  
\end{proof}

\begin{remark}
We saw that pseudo-rational schemes with a canonical sheaf are birational derived splinters by 
\autoref{lem.BirationalSplinterImpliesBirationallyPure}, and hence such $X$ can be used in \autoref{lem.BirClosureContainment}.  Furthermore, if $X \to \Spec R$ is a resolution of singularities with $X$ of finite dimension, one can use that $X$ in \autoref{lem.BirClosureContainment}.  In fact, the finite dimensional hypothesis is not needed to prove that regular $X$ is a derived splinter via an argument of Bhatt (\cite[Remark 3.4]{MaSchwede.KunzTypeCharacterization}, \cite[Lemma 3.16]{LankVenkatesh.TriangulatedCategoriesOfSingularities}), we shall repeat Bhatt's argument within the proof of \autoref{thm.UniformBS} below.  Recall resolutions of singularities are known to exist if $R$ is excellent and either $\dim R \leq 3$ \cite{Abhyankar.ResOfSingsOfEmbeddedSurfaces,Lipman.Desingularization,Cutkosky.ResOfSingularitiesBook,CossartPiltant.ResInCharp1,CossartPiltant.ResInCharp2,CossartPiltant.ResOfSingsOfArithmeticalThreefolds} or $R$ has characteristic zero \cite{HironakaResolution,Tempkin.DesingularizationQuasiExcellentChar0}.
\end{remark}


It is not clear to us whether or not $J^{\Bir}$ is a closure operation.  In other words:

\begin{question}  With notation as above, is $(J^{\Bir})^{\Bir} = J^{\Bir}$? \end{question}

However, we obtain the following result.
\begin{proposition}
\label{prop.BirClosureVersionOfBS}
    Suppose $R$ is a reduced Noetherian ring and $J = (f_1, \dots, f_n)$ is an ideal.  Then $\overline{J^{n+k-1}} \subseteq (J^k)^{\Bir}$ for any $k \geq 1$.  
    The same holds if $J \subseteq R$ is an ideal that locally has analytic spread $\leq n$.  
\end{proposition}
\begin{proof}
As both integral closure and birational closure can be computed modulo minimal primes we may assume that $R$ is a domain.  If $J = 0$ then the statement is clear.
Set $Y \to \Spec R$ to be the blowup of $\overline{J^{n+k-1}}$.  This map is proper and birational since $R$ is a domain and $J$ is nonzero.
   We see there is a map from $L^k(\underline{f})$ to its bottom homology, $L^k(\underline{f}) \to R/J^k$.  Hence we see 
   \[
    \overline{J^{n+k-1}} \subseteq \ker\Big(R \to H_0\big( L^k(\underline{f}) \otimes^{\myL} \myR\Gamma(Y, \cO_Y) \big) \Big) \subseteq \ker\Big(R \to H_0\big( R/J^k \otimes^{\myL} \myR\Gamma(Y, \cO_Y) \big) \Big) \subseteq (J^k)^{\Bir}.
   \]

   For the second statement, as the statement is local (indeed, the formation of $J^{\Bir}$ commutes with localization), we assume that $R$ is local.  Then one argues again as in \autoref{thm.StrongLipmanTeissier} to replace $J$ by its minimal reduction.  
\end{proof}

\autoref{prop.BirClosureVersionOfBS} implies many other well-known closure-versions of the \myBrianconSkoda theorem.

\subsection{Characteristic zero}

In combination with \autoref{lem.BirClosureContainment}, \autoref{prop.BirClosureVersionOfBS} gives the following answer to \cite[Question 6.7]{EpsteinMcDonaldRGSchwede}.

\begin{cor}
    Let $R$ be an excellent ring of characteristic zero and $J=(f_1,\dots,f_n)$.  Then 
    \[
        (J^k)^{\Bir} = (J^k)^{\Hir}  := \ker\big(R \to H_0(R/J^k \otimes^{\myL} \myR\Gamma(X, \cO_X))\big)
    \]
    where $X \to \Spec R$ is any resolution of singularities.
    Hence  
    $\overline{J^{n+k-1}}\subseteq(J^k)^{\Hir}
    $
    for all $k\geq1$.
\end{cor}

Statements for other characteristic zero closures follow from the characteristic $p > 0$ results below.

\subsection{Characteristic $p > 0$} 

Suppose $R$ is a reduced Noetherian ring of characteristic $p > 0$.  Let $R^+ = \prod_Q (R/Q)^+$ where $Q$ runs over the minimal primes of $R$.  For any ideal $J \subseteq R$, we denote by $J^+ := (J R^+) \cap R$.  Note $J^+ \subseteq J^*$ by \cite[Lemma 4.11]{HHmain}.

\begin{proposition}[{\cite[Theorem 7.1]{HochsterHunekeApplicationsofBigCM}}]
\label{prop.CharPContainment}
    Suppose $R$ is a reduced Noetherian ring and $I \subseteq R$ is any ideal.  Then $I^{\Bir} \subseteq I^+$.  As a consequence, if $J = (f_1, \dots, f_n) \subseteq R$, then we have  
    \[
        \overline{J^{{n+k-1}}} \subseteq (J^k)^+
    \]
    for every integer $k > 0$.
    {Hence we also obtain $\overline{J^{{n+k-1}}} \subseteq (J^k)^*$ recovering \cite[Theorem 5.4]{HHmain}.}
\end{proposition}
\begin{proof}
As $R$ is Noetherian, we may fix $\pi : Y \to X$ proper birational computing $J^{\Bir}$. 
    By \cite[Theorem 0.4]{BhattDerivedDirectSummandCharP}, there is a map $\myR\Gamma(Y, \cO_Y) \to R^+$ factoring the inclusion $R \to R^+$.
    Hence we immediately see that 
    \[
        \ker\Big(R \to H_0(R/J^k \otimes^{\myL} \myR\Gamma(Y, \cO_Y)) \Big) \subseteq \ker\Big(R \to R/J^k \otimes R^+ \Big) = (J^k)^+.
    \]
    The second and third statements follow immediately.
\end{proof}

\subsection{Mixed characteristic}

The same argument goes through without substantial change in the mixed characteristic case thanks to \cite{BhattAbsoluteIntegralClosure}. In fact, \autoref{prop.MixedCharComparisonToClosure} below strictly contains \autoref{prop.CharPContainment} since the $p$-adic completion is a trivial operation in characteristic $p$.

Let $R$ denote a Noetherian reduced ring with $p$ in its Jacobson radical.  By $\widehat{R^+}$ we mean the $p$-adic completion of $R^+ = \prod_Q (R/Q)^+$ where again $Q$ runs over the minimal primes of $R$.

\begin{proposition}[{\cite{RodriguezVillalobosSchwede.BrianconSkodaProperty, heitmannepf}}]
\label{prop.MixedCharComparisonToClosure}
    Suppose $R$ is a reduced Noetherian ring with $p$ in its Jacobson radical and suppose that $I \subseteq R$ is an ideal.  Then $I^{\Bir} \subseteq I\widehat{R^+} \cap R \subseteq I^{\mathrm{epf}}$ where $\widehat{R^+}$ denotes the $p$-adic completion of $R^+$ and $\epf$ denotes the full extended plus closure. 
    
    As a consequence, if $J = (f_1, \dots, f_n) \subseteq R$, then
    \[
        \overline{J^{{n+k-1}}} \subseteq (J^k) \widehat{R^+} \cap R
    \] 
    recovering \cite[Theorem 4.1]{RodriguezVillalobosSchwede.BrianconSkodaProperty}.  Hence in an arbitrary ring, we also recover $\overline{J^{n+k-1}} \subseteq (J^k)^{\ep}$ of \cite[Theorem 4.2]{heitmannepf} where $\ep$ denotes the extended plus closure.
\end{proposition}
\begin{proof}
    We may assume that $R$ is a domain as all objects can be computed mod minimal primes.  We then assume that $R$ has mixed characteristic $(0, p>0)$ as the equal characteristic $p > 0$ case was done above.  
    Suppose $Y \to \Spec R$ is proper birational.  By 
    \cite[Theorem 3.12]{BhattAbsoluteIntegralClosure} and derived Nakayama, the derived completion of $\myR\Gamma(Y^+, \cO_{Y^+})$ coincides with $\widehat{R^+}$.  Hence we have a factorization 
    \[
        R \to \myR\Gamma(Y, \cO_Y) \to \widehat{R^+}.
    \]
    The first result follows.  Note $I \widehat{R^+} \cap R \subseteq I^{\epf}$ essentially by definition of full extended plus closure in \cite{heitmannepf}. For the statement about  extended plus closure $(-)^{\ep}$, by using \cite[Proposition 1.2 and Lemma 2.3]{heitmannepf} we may reduce to the case of a local domain essentially of finite type over $\bZ$ with some prime $p > 0$ in its maximal ideal.  In particular, $\ep$ coincides with $\epf$ in this case.
\end{proof}

\section{Uniform Brian\c{c}on-Skoda and uniform Artin-Rees}
\label{sec.UniformBS}
In this section, we prove two conjectures of Huneke \cite[Conjecture 1.3 and Conjecture 1.4]{Huneke.UniformBoundsInNoetherianRings} as a major application of our results and methods in previous sections.  
We first give an outline of the strategy. Suppose $R$ is a domain and $\pi : X \to \Spec R$ is a proper surjective map with $X$ an integral global birational derived splinter (e.g., $\pi$ could be a regular alteration). Suppose we have a map $\psi : \myR\Gamma(X, \cO_X) \to R$ such that the composition 
$$R \to \myR\Gamma(X, \cO_X) \xrightarrow{\psi} R$$ 
is multiplication by $c \neq 0$ (such maps always exist since $\myR\Gamma(X, \cO_X) \in D^b_{\coh}(R)$ and since such maps exist generically).
Let $I = (f_1, \dots, f_n)\subseteq R$ and set $Y' \to X$ to be the blowup of $\overline{I^{n+k-1}} \cO_X$ (so that the composition $Y' \to \Spec R$ factors through $Y \to \Spec R$, the blowup of $\overline{I^{n+k-1}}$), then the argument from \autoref{rem.SplittingByMultByC} immediately implies that we have 
    $
        c \overline{I^{n+k-1}} \subseteq I^k.
    $  
Of particular importance here is that $c$ does not depend on $I$ and $k$. The existence of such a uniform $c$, which forces $\mathrm{T}_n(R) := \bigcap_{m,I} (I^{m - n} : \overline{I^m})$ to be nonzero for some $n$ (which could be taken as $\dim R$ eventually, by our method), is one of Huneke's key tools for proving uniform \myBrianconSkoda and uniform Artin-Rees properties in \cite{Huneke.UniformBoundsInNoetherianRings}.

It was known, for rings that are quotients of regular rings, that the existence of projective resolutions of singularities implies these uniformity statements via a different argument, see  \cite{Huneke.DesingsAndUniformArtinRees}. Furthermore, in unpublished work, Datta showed that Huneke's argument also goes through assuming only that projective regular alterations exist \cite{DattaPrivateCommunication}.  

However, the existence of regular alterations is not known in full generality for excellent domains.  Instead, we adapt the argument above to use Gabber's weak local uniformization (\cite[Expos\'{e} VII, Theorem 1.1]{GabberWeakAlteration}), and instead of using regular alterations, we use a certain diagram of schemes.
The other key ingredient in Huneke's recipe, that $0 \neq \mathrm{CM}(R)$, where $\mathrm{CM}(R)$ is the elements of $R$ that annihilate the homology of complexes of finitely generated free modules satisfying the standard conditions on height and rank, was proven in \cite{ZhouUniformAnnihilators}. Alternately, this also follows quickly by combining \cite{LyuFormalLifting} and \cite{KawasakiMacaulayfication} up to a faithfully flat extension which is enough for us.

We begin with a lemma about truncated hypercovers in the alteration topology.

\begin{lemma}
\label{lem.TruncatedAlteration}
    Let $X$ be an irreducible, quasi-excellent Noetherian scheme. Then for every integer $n\geq 0$, we can construct a commutative diagram 
    \[
        \xymatrix{
            {^n}V_{\leq n} \ar[r] \ar[d] & {^n}V_\bullet  \ar[r] & {^n}V\ar[d]\\
            X_{\leq n} \ar[r] & X_\bullet \ar[r] & X
        }
        \]
    where ${^n}V\to X$ is an alteration, ${^n}V_{\bullet}\to {^n}V$ is a Zariski hypercover, and $X_{\bullet}\to X$ is a hypercover in the alteration topology with each $X_i$ regular. Here $V_{\leq n}$ and $X_{\leq n}$ are the $n$-truncations of ${^n}V_\bullet$ and $X_\bullet$ respectively.
\end{lemma}

\begin{proof}

We will repeatedly use Gabber's weak local uniformization theorem \cite[Expos\'{e} VII, Theorem 1.1]{GabberWeakAlteration}. We begin by selecting a cover $\{X_0^i\to X\}$ in the alteration topology so that each $X_0^i$ is regular by weak local uniformization. By 
\cite[Expos\'{e} II, Theorem 3.2.1]{GabberWeakAlteration}, there exists an alteration ${^0}V\to X$ and a {Zariski} open cover ${^0}V^i$ of ${^0}V$ which fits into the following commutative diagram
\[\xymatrix{
{^0}V^{i} \ar[r] \ar[d] & {^0}V \ar[d] \\
X_0^i \ar[r] & X
.}
\]

Next, we find a cover $\{{^0}Y^j\to {^0}V\}$ in the alteration topology so that each ${^0}Y^j$ is regular by weak local uniformization. Again by \cite[Expos\'{e} II, Theorem 3.2.1]{GabberWeakAlteration}, there exists an alteration ${^1}V\to {^0}V$ and a Zariski open cover ${^1}\widetilde{V}^j$ of ${^1}V$ which fits into the following commutative diagram
\[\xymatrix{
{^1}\widetilde{V}^{j} \ar[r] \ar[d] & {^1}V \ar[d] \\
{^0}Y^j \ar[r] & {^0}V
.}
\]
At this point, we note that each ${^0}V^a \times_{{^0}V} {^0}V^b$ admits a regular cover in the alteration topology
$$\{{^0}Y^j\times_{{^0}V}({^0}V^a \times_{{^0}V} {^0}V^b) \to {^0}V^a \times_{{^0}V} {^0}V^b\}.$$
Now, we set 
${^1}V^i:= {^0}V^i\times_{{^0}V}{^1}V$ (a Zariski open subset of ${^1}V$).  By the commutative diagram above and base change, we find that 
$$\{{^1}\widetilde{V}^j \times_{{^1}V}({^1}V^a  \times_{{^1}V} {^1}V^b) \to {^1}V^a \times_{{^1}V} {^1}V^b\}$$
is a Zariski open cover and that each ${^1}\widetilde{V}^j \times_{{^1}V}({^1}V^a \times_{{^1}V} {^1}V^b)$ factors through the regular scheme ${^0}Y^j\times_{{^0}V}({^0}V^a \times_{{^0}V} {^0}V^b)$, which in turn maps to $X_0^a \times^{\text{alt}}_X X_0^b$ (which is the closed subscheme of $X_0^a\times_XX_0^b$ made up of components that dominate $X$, in particular it maps to $X_0^a\times_XX_0^b$).\footnote{The alteration topology admits fiber products, denoted $\times^{\text{alt}}$, thanks to \cite[Expos\'{e} II, Proposition 1.2.6]{GabberWeakAlteration}.  It is exactly the the dominant component(s) of the usual fiber product. Under this fiber product, $\{{^0}Y^j\times_{{^0}V}({^0}V^a \times_{{}^0V}{^0}V^b) \to X_0^a \times^{\text{alt}}_X X_0^b\}$ is a cover in the alteration topology.} In sum, we find that ${^1}V$ has a $1$-truncated Zariski hypercover ${^1}V_{\leq 1} \to {^1}V$ with
$${^1}V_0=\{{^1}V^i\}, \text{ and } {^1}V_1=\{{^1}\widetilde{V}^j \times_{{^1}V}({^1}V^a \times_{{^1}V} {^1}V^b)\}.$$ 
Moreover, this ${^1}V_{\leq 1}$ dominates the following $1$-truncated {
alteration hypercover $X_{\leq 1}$ by regular schemes}, where
$$X_0=\{X^i_0\}, \text{ and } X_1=\{{^0}Y^j\times_{{^0}V}({^0}V^a \times_{{^0}V} {^0}V^b) \}.$$

{Note that as written, the above is not technically a hypercover, as it only has face maps and no degeneracy maps, and hypercovers are simplicial objects. However, using the notion of a \emph{split} simplicial object (see for instance \cite[6.2.2]{DeligneHodgeTheory3}, \cite[ex~Vbis~5.1]{ToposTheoryEtaleCohomology}, and \cite[Remark~2.14]{LeeLocalAcyclicFibrations}), we only need to specify the non-degenerate simplices appearing in any degree in order to define a unique simplicial object. More precisely, in the language of \cite[Theorem~4.12]{ConradCohomologicalDescent}, given a split $n$-truncated simplicial scheme $W$ and a collection $N$ that we would like to adjoin as non-degenerate $(n+1)$-simplices and a map $N\to(\text{cosk}_nW)_{n+1}$, there is a unique, up to isomorphism, split $(n+1)$-truncated simplicial scheme $W'$ such that $W'_{\leq n}=W_{\leq n}$ and the non-degenerate $(n+1)$ simplices of $W'$ are $N$. Thus, we will abuse notation and specify only the nondegenerate simplices in each degree in what follows.}

We will now essentially repeat this procedure to inductively construct ${^m}V$, $m\geq 1$, with an $m$-truncated Zariski hypercover ${^m}V_{\leq m} \to {^m}V$ that dominates an {
$m$-truncated alteration hypercover $X_{\leq m}$} so that each $X_i$ is regular. Indeed, suppose ${^m}V$ is constructed, we use weak local uniformization to obtain a regular cover $\{{^m}Y^j\to {^m}V\}$ in the alteration topology, and find ${^{m+1}}V$ an alteration over ${^m}V$ together with a {Zariski} open cover $\{{^{m+1}}\widetilde{V}^j\}$ dominating $\{{^m}Y^j\}$. The $m$-truncated Zariski hypercover ${^m}V_{\leq m}$ pulls back to an $m$-truncated Zariski hypercover ${^{m+1}}V_{\leq m}$ of ${^{m+1}}V$. For each element $U \in (\text{cosk}_{m}({^{m+1}}V_{\leq m}))_{m+1}$, {noting it is the pullback of some} $U' \in (\text{cosk}_{m}({^{m}}V_{\leq m}))_{m+1}$, consider the open cover
$$\{{^{m+1}}\widetilde{V}^j\times_{{^{m+1}}V}U\to U\}.$$
The collection of all these $\{{^{m+1}}\widetilde{V}^j\times_{{^{m+1}}V}U\}_{U,j}$ thus form a Zariski open cover of $(\text{cosk}_{m}({^{m+1}}V_{\leq m}))_{m+1}$. Moreover, each ${^{m+1}}\widetilde{V}^j\times_{{^{m+1}}V}U$ maps to the regular scheme ${^m}Y^j \times_{{^m}V}U'$. 
In sum, we have 
\[\xymatrix{
\{{^{m+1}}\widetilde{V}^j\times_{{^{m+1}}V}U\}_{U,j} \ar[r]\ar[d] & (\text{cosk}_{m}({^{m+1}}V_{\leq m}))_{m+1} \ar[d] \\
\{{^m}Y^j \times_{{^m}V}U'\}_{U', j} \ar[r] \ar@{.>}[rd]& (\text{cosk}_{m}({^m}V_{\leq m}))_{m+1} \ar[d] \\
& (\text{cosk}_{m}(^{m}X_{\leq m}))_{m+1}
.}
\]
Here $(\text{cosk}_{m}(^{m}X_{\leq m}))_{m+1}$ is computed in the alteration topology (so it is built by fiber products in the alteration topology), and the dotted arrow above is a cover in the alteration topology. We now set ${^{m+1}}V_{m+1} := \{{^{m+1}}\widetilde{V}^j\times_{{^{m+1}}V}U\}_{U,j}$ and $X_{m+1}:= \{{^m}Y^j \times_{{^m}V}U'\}_{U', j}$. It is then readily checked that we have the following commutative diagram
\[\xymatrix{
{^{m+1}}V_{\leq m+1} \ar[r] \ar[d] & {^{m+1}}V \ar[d] \\
X_{\leq m+1} \ar[r] & X 
.}
\]

Finally, we set ${^n}V_\bullet := \text{cosk}_{n}({}^nV_{\leq n})$ and we extend $X_{\leq n}$ to {
an alteration hypercover $X_\bullet$} so that each $X_i$ is regular by weak local uniformization. Note that ${}^nV_\bullet$ is a Zariski hypercover of ${}^nV$. We obtain the desired commutative diagram
\[\xymatrix{
{}^nV_{\leq n} \ar[r] \ar[d] & {}^nV_\bullet  \ar[r] & {}^nV\ar[d]\\
X_{\leq n} \ar[r] & X_\bullet \ar[r] & X
.}
\]
Note that there might not be a map from ${}^nV_\bullet\to X_\bullet$ fitting the commutative diagram above though, since we built ${}^nV_\bullet$ out of $V_{\leq n}$ while one might need further simplices to build $X_{\bullet}$ to keep the terms nonsingular.
\end{proof}

\begin{remark}
In the proof of \autoref{lem.TruncatedAlteration}, instead of writing a Zariski open cover $\{{^{m+1}}\widetilde{V}^j\}$ in each step, we could use a single cover ${^{m+1}}\widetilde{V}:= \coprod_j{^{m+1}}\widetilde{V}^j$ to reduce the excess of indices and simplify the  notation (the proof will not be altered too much). We decided to write down the open cover and keep the additional index $j$ to make the construction clear and to remind ourselves that each ${^{m+1}}\widetilde{V}^j$ dominates a ${^{m}}{Y}^j$ (which is based on Gabber's weak local uniformization).
\end{remark}

We are now ready to prove our main result which will be the key ingredient in proving Huneke's conjecture.

\begin{theorem}
\label{thm.UniformBS}
Let $R$ be an excellent domain of dimension $d$. Then ${\mathrm T}_{d}(R) \neq 0$.  
\end{theorem}
\begin{proof}
    First, note that we can replace $R$ by $R\otimes_{\Z}\Z(t)$ {(where $\bZ(t)$ denotes $\bZ[t]$ localized at the multiplicative set of all elements not contained in any $p \bZ[t]$)} to assume that all residue fields of $R$ are infinite. 
    
    Showing that $\text{T}_{d}(R)=\bigcap_{I, n}(I^{n-d}: \overline{I^n})$ is nonzero is the same as showing that $\bigcap_{I, k\geq 1}(I^{k}: \overline{I^{d+k}})$ is nonzero. Since $R$ has infinite residue fields, every ideal $I\subseteq R$ admits a minimal reduction $J$ generated by at most $d+1$ elements by \cite[Theorem 8.7.3]{SwansonHuneke.integralclosure}. By \autoref{thm.DerivedMain}, if $Y$ is the normalized blowup of $J$ (which agrees with the blowup of $\overline{J^N}$ for some $N\gg0$), then the natural map 
        $$\overline{I^{d+k}} \to R\to \myR\Gamma(Y, \cO_Y) \to L^k(J)\otimes^{\myL} \myR\Gamma(Y, \cO_Y)$$
    is the zero map in $D(R)$.
    
    Let $X=\Spec R$. Applying \autoref{lem.TruncatedAlteration} with $n=d+2$ and setting $V:={}^{d+2}V$ and $V_\bullet:= {}^{d+2}V_\bullet$, we have a commutative diagram 
        \[
        \xymatrix{
            V_{\leq d+2} \ar[r] \ar[d] & V_\bullet  \ar[r] & V\ar[d]\\
            X_{\leq d+2} \ar[r] & X_\bullet \ar[r] & X
        }
        \]
        where $V_\bullet$ is a Zariski hypercover of $V$, $V\to X$ is an alteration, and $X_\bullet\to X$ is an alteration hypercover of $X$ with each $X_i$ regular. 
        
    Let $Y_i$ be the blowup of $\overline{J^N}\cO_{X_i}$ for each $X_i$. We have a commutative diagram of diagrams of schemes:
        \[
        \xymatrix{
        Y_\bullet \ar[r] \ar[d]  & Y \ar[d] \\
        X_\bullet \ar[r] & X
        }
        \]
    which induces a commutative diagram: 
        \[
        \xymatrix{
        \overline{I^{d+k}} \ar[r] & R \ar[d] \ar[r] & \myR\Gamma(Y, \cO_Y) \ar[r] \ar[d] & L^k(J)\otimes^{\myL} \myR\Gamma(Y, \cO_Y) \ar[d] \\
        & \myR\Gamma(X_\bullet, \cO_{X_\bullet}) \ar[r] & \myR\Gamma(Y_\bullet, \cO_{Y_\bullet}) \ar[r] &  L^k(J)\otimes^{\myL} \myR\Gamma(Y_\bullet, \cO_{Y_\bullet}).
        }
        \]
    In particular, the diagram guarantees that the natural map 
        $$\overline{I^{d+k}} \to R\to \myR\Gamma(X_\bullet, \cO_{X_\bullet})  \to L^k(J)\otimes^{\myL} \myR\Gamma(Y_\bullet, \cO_{Y_\bullet})$$
    is zero in $D(R)$.

    We now repeat an argument of Bhatt's (\cite[Remark 3.4]{MaSchwede.KunzTypeCharacterization}, \cite[Lemma 3.16]{LankVenkatesh.TriangulatedCategoriesOfSingularities}).  We note that since each $X_i$ is regular, each $\myR{\pi_i}_*\cO_{Y_i} \in D_{\text{perf}}(X_i)$ and is a commutative algebra object in $D(X_i)$ (here $\pi_i$ denotes the map $Y_i\to X_i$). Therefore we have canonical maps 
        $$\cO_{X_i} \to \myR{\pi_i}_*\cO_{Y_i} \xrightarrow{s} \myR\sHom_{\cO_{X_i}}(\myR{\pi_i}_*\cO_{Y_i}, \myR{\pi_i}_*\cO_{Y_i}) \xrightarrow{t} \cO_{X_i}$$
    where $s$ comes from the commutative algebra structure on $\myR{\pi_i}_*\cO_{Y_i}$ and $t$ is the trace map coming from the perfectness of $\myR{\pi_i}_*\cO_{Y_i}$. Moreover, the composition map is generically the identity and hence the identity map (note that by construction, all the (irreducible components of) $X_i$'s and $Y_i$'s are generically finite over and dominate $X = \Spec R$). As $s$ and $t$ are functorial, these induce maps
        $$\cO_{X_\bullet} \to \myR{\pi_\bullet}_*\cO_{Y_\bullet} \to \cO_{X_\bullet}$$
    such that the composition is the identity.  
    It follows that the canonical map $\myR\Gamma(X_\bullet, \cO_{X_\bullet}) \to \myR\Gamma(Y_\bullet, \cO_{Y_\bullet})$ splits in $D(R)$. As a consequence, the natural map 
        $$\overline{I^{d+k}} \to R\to \myR\Gamma(X_\bullet, \cO_{X_\bullet})  \to L^k(J)\otimes^{\myL} \myR\Gamma(X_\bullet, \cO_{X_\bullet})$$
    is zero in $D(R)$. 
    We thus see that the composition
    \begin{align*}
     \overline{I^{d+k}} \to & R\to \myR\Gamma(X_\bullet, \cO_{X_\bullet})  \to L^k(J)\otimes^{\myL} \myR\Gamma(X_\bullet, \cO_{X_{\bullet}}) \\
     & \to L^k(J)\otimes^{\myL} \myR\Gamma(X_{\leq d+2}, \cO_{X_{\leq d+2}})\to L^k(J)\otimes^{\myL} \myR\Gamma(V_{\leq d+2}, \cO_{V_{\leq d+2}}) 
    \end{align*}
    is zero in $D(R)$. Since $J$ is generated by at most $d+1$ elements, the Buchsbaum-Eisenbud complex $L^k(J)$ has length at most $d+1$. It follows that (since $V_\bullet\to V$ is a Zariski hypercover)
    \[
        H_0\big(L^k(J) \otimes^{\myL} \myR\Gamma(V,\cO_{V})\big)\cong H_0\big(L^k(J) \otimes^{\myL} \myR\Gamma(V_{\bullet},\cO_{V_{\bullet}})\big) \cong H_0\big(L^k(J) \otimes^{\myL} \myR\Gamma(V_{\leq d+2},\cO_{V_{\leq d+2}})\big)
    \]
    where the second isomorphism follows from a simple spectral sequence argument.
    Thus after taking $H_0(-)$, we find that the map 
        $$\overline{I^{d+k}} \to H_0\big(L^k(J) \otimes^{\myL} \myR\Gamma(V_{\leq d+2},\cO_{V_{\leq d+2}})\big)\cong H_0\big(L^k(J) \otimes^{\myL} \myR\Gamma(V,\cO_{V})\big)$$
    is zero. However, since $V$ is proper surjective over $X$, we can pick $c\neq 0$ so that multiplication by $c$ on $R$ factors through $\myR\Gamma(V,\cO_{V})$  (since $R\to \myR\Gamma(V,\cO_{V})$ splits after tensoring with the fraction field of $R$ and $\myR\Gamma(V,\cO_{V})\in D^b_{\text{coh}}(R)$). Note that $c$ depends only on $V$ and does not depend on $I$ or $k$. We thus find that $c\overline{I^{d+k}}\subseteq J^k\subseteq I^k$ for all $I$ and all $k\geq 1$, i.e., $\text{T}_{d}(R)\ni c\neq 0$ as desired.
\end{proof}

\begin{corollary}[Uniform \myBrianconSkoda]
\label{cor.UniformBS}
    Suppose $R$ is a quasi-excellent reduced Noetherian ring of finite dimension.  Then there exists a positive integer $k$ such that for all ideals $I\subseteq R$, $\overline{I^n}\subseteq I^{n-k}$ for all integers $n\geq k$.
\end{corollary}
\begin{proof}
    First of all, we can replace $R$ by $R\otimes_{\Z}\Z(t)$ {(where $\bZ(t)$ denotes $\bZ[t]$ localized at multiplicative set of all elements not contained in any $p \bZ[t]$)} to assume that all residue fields of $R$ are infinite. Secondly, by replacing $R$ by a faithfully flat \'{e}tale cover, we may assume that $R$ is excellent by \cite[Corollary 6.6]{LyuFormalLifting} (Lyu pointed out to us this generalization from excellent to quasi-excellent).

    By \cite[Proposition 3.7]{Huneke.UniformBoundsInNoetherianRings}, it is enough to show that $\text{CM}(R/P)\neq 0$ and that $\text{T}(R/P) := \bigcup_k \mathrm{T}_k(R/P) \neq0$ for all prime ideals $P\in\Spec(R)$. We have just shown in \autoref{thm.UniformBS} that $\text{T}(R/P)\neq0$. We know $\text{CM}(R/P) \neq 0$ by \cite{ZhouUniformAnnihilators}. Alternately, by replacing $R$ by a faithfully flat \'{e}tale cover (which we already may have done when reducing to the excellent case), we may assume that $R$ admits a dualizing complex by \cite[Theorem 6.5]{LyuFormalLifting}, and thus $R/P$ is a homomorphic image of a Gorenstein ring of finite dimension by \cite[Theorem 1.2]{KawasakiMacaulayfication} and so $\text{CM}(R/P)\neq 0$ by \cite[Proposition 4.5 (i)]{Huneke.UniformBoundsInNoetherianRings}.
\end{proof}

\begin{remark}
For any Noetherian local ring $(R,\m)$ of dimension $d$ and any system of parameters $x_1,\dots,x_d$, it is easy to see that 
$$\overline{(x_1^d,\dots,x_d^d)^{d-1}}\nsubseteq (x_1^d,\dots,x_d^d).$$
In particular, one cannot expect the conclusion of \autoref{cor.UniformBS} to hold for infinite dimensional excellent rings (such examples exist, see \cite{TanakaInfiniteDimensionalExcellent}): for if such $k$ exists, then localizing at a prime of height $k+1$ would contradict the non-containment above.
\end{remark}

\begin{corollary}[Uniform Artin-Rees]
\label{cor.UniformAR}
    Suppose $R$ is a quasi-excellent Noetherian ring of finite dimension and $N \subseteq M$ are finitely generated $R$-modules.  Then there exists an integer $\ell$ depending on $N$ and $M$ such that for all ideals $I\subseteq R$ and all $n \geq \ell$, we have that 
    \[
        I^n M \cap N \subseteq I^{n-\ell} M.
    \]
\end{corollary}
\begin{proof}
    By base change, we can again replace $R$ by $R\otimes_{\Z}\Z(t)$ and so assume $R$ has infinite residue fields and, as above, assume that $R$ is excellent by \cite[Corollary 6.6]{LyuFormalLifting}.  For all $P\in\Spec R$, we know that ${\mathrm{CM}}(R/P) \neq 0$ by \cite{ZhouUniformAnnihilators} (or alternately as above).  By \autoref{thm.UniformBS}, we see that $\mathrm{T}(R/P) \neq 0$.  Hence the result follows from \cite[Theorem 3.4]{Huneke.UniformBoundsInNoetherianRings}.
\end{proof}

\bibliographystyle{skalpha}
\bibliography{alterationsrefs}

\end{document}